\documentclass[a4paper]{article}

\usepackage{amsmath,amsthm,amssymb}
\usepackage{bbold}
\usepackage{amsmath,latexsym,amsfonts,enumerate}
\usepackage{hyperref}
\usepackage{mathdots}
\usepackage{mathtools}
\usepackage{multicol}
\usepackage{nicematrix}
\usepackage{cleveref}
\usepackage{subcaption}
\usepackage{tikz-cd}
\usepackage{titlesec}
\usepackage{xcolor}

\newtheorem{theorem}{Theorem}
\newtheorem{lemma}[theorem]{Lemma}

\newtheorem{proposition}[theorem]{Proposition}
\newtheorem{corollary}[theorem]{Corollary}

\theoremstyle{remark}
\newtheorem{remark}[theorem]{Remark}
\newtheorem{example}[theorem]{Example}

\theoremstyle{remark}
\newcommand{\thistheoremname}{}
\newtheorem*{generictheorem}{\thistheoremname}

\newcommand{\calA}{\mathcal{A}}

\newcommand{\calH}{\mathcal{H}}

\newcommand{\calL}{\mathcal{L}}

\newcommand{\calN}{\mathcal{N}}
\newcommand{\calO}{\mathcal{O}}

\newcommand{\calR}{\mathcal{R}}

\newcommand{\ignore}[1]{}

\newcommand{\bbN}{\mathbb{N}}

\newcommand{\bbR}{\mathbb{R}}

\newcommand{\bfnull}{\boldsymbol{0}}

\newcommand{\bfa}{\mathbf{a}}
\newcommand{\bfb}{\mathbf{b}}

\newcommand{\bfx}{\mathbf{x}}
\newcommand{\bfy}{\mathbf{y}}

\newcommand{\bfA}{\mathbf{A}}
\newcommand{\bfB}{\mathbf{B}}

\newcommand{\bfD}{\mathbf{D}}

\newcommand{\bfP}{\mathbf{P}}

\newcommand{\bfW}{\mathbf{W}}
\newcommand{\bfX}{\mathbf{X}}
\newcommand{\bfY}{\mathbf{Y}}

\newcommand{\rmc}{\mathrm{c}}

\DeclarePairedDelimiterX{\set}[2]{\lbrace}{\rbrace}{#1 \,\delimsize\vert\, #2}
\DeclarePairedDelimiterX{\abs}[1]{\lvert}{\rvert}{#1}
\DeclarePairedDelimiterX{\norm}[1]{\lVert}{\rVert}{#1}

\title{Stability of sorting based embeddings}
\author{Radu Balan, Efstratios Tsoukanis and Matthias Wellershoff}
\date{\today}

\begin{document}
\maketitle

\begin{abstract}
    Consider a group $G$ of order $M$ acting unitarily on a real inner product space $V$. We show that the sorting based embedding obtained by applying a general linear map $\alpha : \bbR^{M \times N} \to \bbR^D$ to the invariant map $\beta_\Phi : V \to \bbR^{M \times N}$ given by sorting the coorbits $(\langle v, g \phi_i \rangle_V)_{g \in G}$, where $(\phi_i)_{i=1}^N \in V$, satisfies a bi-Lipschitz condition if and only if it separates orbits.
    
    Additionally, we note that any invariant Lipschitz continuous map (into a Hilbert space) factors through the sorting based embedding, and that any invariant continuous map (into a locally convex space) factors through the sorting based embedding as well.
\end{abstract}

\section{Introduction}

\newcommand{\group}{G}
\newcommand{\vsp}{V}
\newcommand{\dimvsp}{d}
\DeclarePairedDelimiterX{\innervsp}[2]{\langle}{\rangle_\vsp}{#1, #2}
\DeclarePairedDelimiterX{\normvsp}[1]{\lVert}{\rVert_\vsp}{#1}
\newcommand{\distqsp}{\operatorname{dist}}
\newcommand{\windows}{\Phi}
\newcommand{\coorbit}[1]{\kappa_{#1}}
\DeclarePairedDelimiterX{\bracket}[1]{(}{)}{#1}
\newcommand{\ordgroup}{M}
\newcommand{\sortedcoorbits}[1]{\beta_{#1}}
\DeclarePairedDelimiterXPP\sort[1]{\operatorname{sort}}(){}{#1}
\newcommand{\nwin}{N}
\newcommand{\linear}{\alpha}
\newcommand{\embeddingdim}{D}
\newcommand{\embedding}{\gamma}

Suppose that a specific learning task with input space $\vsp$ is invariant under the action of a group $\group$. Then, it makes sense to construct a class of hypotheses that factor into two parts: i.~a group invariant part $h : \vsp \to \vsp_\mathrm{int}$, where $\vsp_\mathrm{int}$ is some intermediate space, and ii.~an ``optimisable'' part $g : \vsp_\mathrm{int} \to W$, where $W$ denotes the space of all possible outputs. In this way, we can use an optimisation algorithm to find $g$ in such a way that $f = g \circ h$ fits some training data, and the above construction ascertains that $f : \vsp \to W$ is invariant under the action of $G$.

So, introducing the equivalence relation $v \sim w :\iff \exists g \in G : v = gw$ on $V$ and assuming that $V$ has a norm $\norm{\cdot}_V$, our goal becomes to construct maps $ h : V \to \mathbb{R}^D $, where $ D \in \mathbb{N} $ is as small as possible, and $ h $ satisfies the properties:
\begin{enumerate}
    \item \emph{Invariance.} $ h(v) = h(w) $ for all $ v, w \in V $ such that $ v \sim w $.
    \item \emph{Orbit separation.} $ v \sim w $ for all $ v, w \in V $ such that $ h(v) = h(w) $.
    \item \emph{Bi-Lipschitz condition.} There exist constants $ 0 < c \leq C $ such that 
    \begin{equation}\label{eq:biLipschitz_condition}
        c \distqsp(v,w) \leq \norm{h(v) - h(w)}_2 \leq C \distqsp(v,w), \qquad v, w \in V.
    \end{equation}
\end{enumerate}

\noindent
Here and throughout the paper, $\distqsp(v,w) := \min_{g \in \group} \normvsp{v - gw}$ denotes the natural metric on the quotient space $\vsp/{\sim}$.

The approach described above is an instance of \emph{invariant machine learning}: techniques designed to ensure that hypotheses are robust to specific changes in the input data.
More precisely, it is a special form of \emph{feature engineering}, where raw data is transformed into a more useful set of inputs. 

Feature transforms have been suggested in computer vision, for example, where handcrafted maps such as the scale-invariant feature transform (SIFT) \cite{lowe1999object} or the histogram of oriented gradients (HOG) \cite{dalal2015histograms} were created to be invariant to transformations like scaling, translation or rotation. The results we present here were inspired by the two more recent papers \cite{cahill2024group,dym2024low} and are a continuation of work presented in \cite{balan2022permutation,balan2023GI,balan2023GII}. Similar approaches can also be found in \cite{aslan2023group,cahill2024stable,cahill2024towards} and, most notably, \cite{mixon2024stable}.

Taking a wider view of the literature, there is another notable approach to invariant machine learning which consists of propagating the invariance of a problem through multiple \emph{equivariant layers}: maps $f : V \to W$ such that $f(gv) = g f(v)$ for all $g \in G$ and all $v \in V$. (Here, $G$ is also assumed to act on $W$.) The posterchildren for equivariant machine learning models are convolutional neural networks (CNNs) \cite{lecun1998gradient} which are translation invariant. (Though some care has to be taken when defining ``translation invariance'' \cite{azulay2019why}.)

A well-known generalisation of the CNN architecture in the same spirit is the group equivariant convolutional network architecture \cite{cohen2016group} which introduces layers that can respect more general symmetries. Alternative approaches include \cite{villar2021scalars} (cf.~also \cite{blum-smith2023machine}) as well as \cite{yarotsky2022universal}.

\subsection{Sorting-based embeddings}

Let $ \group $ be a finite group acting unitarily on a $\dimvsp \in \bbN$ dimensional real inner product space $ \vsp $. One approach to construct maps that satisfy items~1 through 3 is to enumerate the group $ \group = \{g_i\}_{i=1}^M $ and define the coorbits
\begin{equation*}
    \kappa_\phi : \vsp \to \mathbb{R}^M, \qquad \kappa_\phi v := \begin{pmatrix}
        \innervsp{ v }{ g_1 \phi } & \dots & \innervsp{ v }{ g_M \phi }
    \end{pmatrix}^\top,
\end{equation*}
for $ \phi \in \vsp $, where $\innervsp{\cdot}{\cdot}$ denotes the inner product on $V$. By choosing a finite sequence $ \Phi := (\phi_i)_{i=1}^N \in \vsp $, sorting the coorbits, and collecting them in a matrix, one obtains an invariant map
\begin{equation*}
    \beta_\Phi : \vsp \to \mathbb{R}^{M \times N}, \qquad \beta_\Phi(v) := \begin{pmatrix}
        \operatorname{sort}(\kappa_{\phi_1} v) & \dots & \operatorname{sort}(\kappa_{\phi_N} v)
    \end{pmatrix},
\end{equation*}
where $\operatorname{sort} : \bbR^M \to \bbR^M$ denotes the operator that sorts vectors in a monotonically decreasing way. To reduce the embedding dimension, a linear map $ \alpha : \mathbb{R}^{M \times N} \to \mathbb{R}^\embeddingdim $ can be applied to the sorted coorbits.

This idea was first introduced in \cite{balan2022permutation} in the context of the action of the group $S_m$ by row permutation on the space of matrices $\bbR^{m \times n}$. The authors demonstrated that the map $ \gamma(\bfX) := \operatorname{sort}(\bfX \bfA) $\footnote{Here, $\operatorname{sort} : \bbR^{m \times N} \to \bbR^{m \times N}$ is the operator that sorts matrices column-wise.} separates orbits for full spark matrices $ \bfA \in \mathbb{R}^{n \times N} $ with $ N > (n-1)m! $. They also showed that $ \gamma $ satisfies the bi-Lipschitz condition (inequality~\eqref{eq:biLipschitz_condition}) if it separates orbits, and that $ \alpha' \circ \gamma $ satisfies inequality~\eqref{eq:biLipschitz_condition} for a generic\footnote{By \emph{generic}, we mean that a statement is true in a non-empty Zariski open set.} linear map $ \alpha' : \mathbb{R}^{m \times N} \to \mathbb{R}^\embeddingdim $ with $ N \geq 2n $ and $ \embeddingdim \geq 2mn $ (provided that $\gamma$ separates orbits). Notably, the embedding $ \alpha' \circ \gamma $ passes through an intermediate Euclidean space of dimension larger than $ (n-1)m! $, which grows rapidly with the matrix size.

The papers \cite{cahill2024group,dym2024low} address this issue. In \cite{dym2024low}, among other things, it is shown that $ \gamma(\bfX) := \operatorname{diag}(\bfB^\top \operatorname{sort}(\bfX \bfA)) $ separates orbits for generic pairs $ (\bfA, \bfB) \in \mathbb{R}^{n \times \embeddingdim} \times \mathbb{R}^{\embeddingdim \times m} $ with $ \embeddingdim > 2mn $. Whether $ \gamma $ satisfies inequality~\eqref{eq:biLipschitz_condition} remained an open question. In \cite{cahill2024group}, the authors show that, if $ \alpha(\bfX) := \begin{pmatrix} X_{11} & \dots & X_{1\embeddingdim} \end{pmatrix} $ selects the maximal entries of the coorbits, then $ \gamma := \alpha \circ \beta_\Phi : \mathbb{R}^d \to \mathbb{R}^\embeddingdim $ separates orbits for generic sequences $ \Phi = (\phi_i)_{i=1}^\embeddingdim \in \mathbb{R}^d $ with $ \embeddingdim \geq 2d $, whenever $ \group \leq O(d) $ is a finite subgroup of the orthogonal $ d \times d $ matrices. They also prove that $ \gamma $ satisfies inequality~\eqref{eq:biLipschitz_condition} with high probability if $ \embeddingdim $ is sufficiently large, and they pose the question whether $ \gamma $ satisfies inequality~\eqref{eq:biLipschitz_condition} whenever it separates orbits.

This was affirmatively answered in \cite{balan2023GI}, where it was shown that $ \gamma = \alpha \circ \beta_\Phi : \vsp \to \mathbb{R}^\embeddingdim $ satisfies the bi-Lipschitz condition if it separates orbits for any finite group $ \group $ acting isometrically on $ \vsp $, provided that $ \alpha : \mathbb{R}^{M \times N} \to \mathbb{R}^\embeddingdim $ selects any subset of the entries of $ \beta_\Phi(v) $. Additionally, in \cite{balan2023GII}, it was shown that such $\gamma$ separate orbits provided that $\Phi$ is chosen generically with respect to the Zariski topology on $\vsp^N$ and that the right subset of entries of $ \beta_\Phi(v) $ is chosen.

\subsection{Results}

In this paper, we generalise the results of \cite{balan2022permutation,balan2023GI} to encompass arbitrary linear maps $ \alpha : \mathbb{R}^{M \times N} \to \mathbb{R}^\embeddingdim $, which proves, in particular, that the embeddings $ \gamma(\bfX) = \operatorname{diag}(\bfB^\top \operatorname{sort}(\bfX \bfA)) $ satisfy inequality~\eqref{eq:biLipschitz_condition} for generic pairs $ (\bfA, \bfB) \in \mathbb{R}^{n \times \embeddingdim} \times \mathbb{R}^{\embeddingdim \times m} $ when $ \embeddingdim > 2mn $.

\begin{theorem}[Main result]\label{thm:main}
    The embedding $ \gamma = \alpha \circ \beta_\Phi : \vsp \to \mathbb{R}^\embeddingdim $ separates orbits if and only if it satisfies the bi-Lipschitz condition~\eqref{eq:biLipschitz_condition}.
\end{theorem}

Moreover, we remark that, as noted in \cite{tsoukanis2024G-invariant}, the hypothesis class of functions that factor into an invariant and an optimisable part, as described at the beginning of the introduction, can contain general invariant (Lipschitz) continuous functions, if $g$ is restricted to a sufficiently large class of functions.

Specifically, we note that any group invariant Lipschitz map into a Hilbert space factors through $\gamma$.

\begin{theorem}\label{thm:cor_kirszbraun}
    Let $f : V \to H$ be invariant under the action of $G$ and Lipschitz continuous with Lipschitz constant $C_f > 0$, where $H$ is some Hilbert space. If $ \gamma = \alpha \circ \beta_\Phi : \vsp \to \mathbb{R}^\embeddingdim $ separates orbits, then there exists $g : \bbR^D \to H$ Lipschitz continuous with Lipschitz constant at most $C_f/c$ such that $f = g \circ \gamma$.
\end{theorem}

\begin{remark}
    It is obvious that any composition of Lipschitz continuous maps is Lipschitz continuous in turn. In particular, $f = g \circ \gamma$ is Lipschitz continuous (and invariant under the action of $G$) provided that $g : \bbR^D \to M$ is Lipschitz continuous, where $M$ is a metric space. 
\end{remark}

Additionally, we note that any continuous group invariant map into a locally convex space also factors through $\gamma$.

\begin{theorem}\label{thm:cor_dugundji}
    Let $f : V \to Z$ be invariant under the action of $G$ and continuous, where $Z$ is some locally convex space. If $ \gamma = \alpha \circ \beta_\Phi : \vsp \to \mathbb{R}^\embeddingdim $ separates orbits, then there exists $g : \bbR^D \to Z$ continuous such that $g(\bbR^D)$ is a subset of the convex hull of $f(V)$ and such that $f = g \circ \gamma$.
\end{theorem}

\begin{remark}
    Vice versa, since compositions of continuous maps are continuous, we have that $f = g \circ \gamma$ is continuous (and invariant under the action of $G$) if $g : \bbR^D \to Z$ is continuous, where $Z$ is a topological space.
\end{remark}

\subsection{Examples}

Finally, we want to remark that our main result does \emph{not} hold for general ReLU neural networks and we want to provide two examples of setups in which our results do apply.

\begin{remark}[There exists an injective ReLU neural networks that is \emph{not} bi-Lipschitz]
    A general ReLU neural network of depth $L \in \bbN$ is a map $\calN : \bbR^{\ell_0} \to \bbR^{\ell_L}$ of the form 
    \begin{equation*}
        \calN(x) := \begin{cases}
            \bfW_1 x + \bfb_1 & \mbox{if } L = 1, \\
            \bfW_L \rho( \dots \rho(\bfW_1 x + \bfb_1) \dots ) + \bfb_L & \mbox{if } L \geq 2,
        \end{cases}
    \end{equation*}
    where $\ell_0,\dots,\ell_L \in \bbN$ denote the layer widths, $\rho : \bbR \to \bbR$, $\rho(x) := \max\lbrace x, 0 \rbrace$, denotes the ReLU activation function, which is applied component-wise to vector inputs, $\bfW_k \in \bbR^{\ell_k \times \ell_{k-1}}$ denote the weight matrices and $\bfb_k \in \bbR^{\ell_k}$ denote the bias vectors for $k \in [L]$.

    A simple example of a map that is injective but not bi-Lipschitz is given by $f : \bbR \to \bbR^2$,
    \begin{equation*}
        f(x) := \begin{cases}
            -(1,x+1) & \mbox{if } x < -1, \\
            (x,0) & \mbox{if } -1 \leq x \leq 1, \\
            (1,x-1) & \mbox{if } 1 < x.
        \end{cases}
    \end{equation*}
    It is not hard to check that $f$ is injective. At the same time, $f$ does not satisfy the lower Lipschitz condition because $\norm{f(x) - f(-x)}_2 = 2$ for $x \not\in [-1,1]$ while $\abs{x - (-x)} = 2 \abs{x}$ is unbounded.
    
    The above map can readily be extended to a map $f_d : \bbR^d \to \bbR^{d+1}$ for $d \in \bbN$ via $f_d(\bfx) = f_d(x_1,x_2,\dots,x_d) = (f(x_1),x_2,\dots,x_d)$. Additionally, $f = f_1$ can be implemented as a ReLU neural network with two layers: indeed, choose 
    \begin{gather*}
        \bfW_1 = \begin{pmatrix}
            1 \\ 1 \\ -1
        \end{pmatrix}, \qquad \bfb_1 = \begin{pmatrix}
            1 \\ -1 \\ -1
        \end{pmatrix}, \\
        \bfW_2 = \begin{pmatrix}
            1 & -1 & 0 \\ 0 & 1 & 1
        \end{pmatrix}, \qquad \bfb_2 = \begin{pmatrix}
            -1 \\ 0
        \end{pmatrix},
    \end{gather*}
    and note that $\calN = f$.
\end{remark}

\begin{example}[Permutation invariant representations]
    Let us consider the action of the group $S_m$ by row permutation on the vector space of matrices $\bbR^{m \times n}$. This setup is naturally encountered in learning on graphs with $m$ vertices, where feature vectors of dimension $n$ are associated to every vertex, if the hypothesis is to be invariant under the permutation of vertices.

    We want to come back to the permutation invariant embedding proposed in \cite{dym2024low} as a variation on the embedding proposed earlier in \cite{balan2022permutation}: $\gamma_{\bfA,\bfB} : \bbR^{m \times n} \to \bbR^D$, 
    \begin{equation*}
        \gamma_{\bfA,\bfB}(\bfX) := \operatorname{diag}(\bfB^\top \operatorname{sort}(\bfX\bfA)), \qquad \bfX \in \bbR^{m \times n},
    \end{equation*}
    where $\bfA \in \bbR^{n \times D}$, $\bfB \in \bbR^{m \times D}$, $\operatorname{diag} : \bbR^{D \times D} \to \bbR^D$ is the linear operator that extracts the diagonal from a $D \times D$ matrix and $\operatorname{sort} : \bbR^{m \times D} \to \bbR^{m \times D}$ denotes column-wise sorting of matrices. 
    
    In \cite{dym2024low}, it is shown that $\gamma_{\bfA,\bfB}$ separates orbits for almost every $(\bfA,\bfB) \in \bbR^{n \times D} \times \bbR^{m \times D}$ provided that $D > 2mn$. Additionally, one may note that $\gamma_{\bfA,\bfB}$ can be evaluated in polynomial time. Indeed, evaluating $\gamma_{\bfA,\bfB}$ requires
    \begin{enumerate}
        \item a matrix multiplication of a matrix of dimension $m \times n$ and a matrix of dimension $n \times D$, which can be performed in $\calO(Dmn)$ operations,
        \item sorting $D$ vectors of length $m$, which can be performed in $\calO(Dm\log m)$ operations,
        \item computing $D$ inner products of vectors of length $m$, which can be performed in $\calO(Dm)$ operations. 
    \end{enumerate}
    Overall, we have an evaluation complexity of $\calO(Dm(n + \log m))$ which, assuming that $D \cong mn$ and $m \cong n$, is $\calO(m^4)$.

    Our main result, Theorem~\ref{thm:main}, shows that $\gamma_{\bfA,\bfB}$ satisfies the bi-Lipschitz condition: there exist constants $0 < c \leq C$ such that 
    \begin{equation*}
        c \min_{\bfP \in S_m} \norm{\bfX - \bfP \bfY}_\mathrm{F} \leq \norm{\gamma_{\bfA,\bfB}(\bfX) - \gamma_{\bfA,\bfB}(\bfY)}_2 \leq C \min_{\bfP \in S_m} \norm{\bfX - \bfP \bfY}_\mathrm{F},
    \end{equation*}
    for all $\bfX,\bfY \in \bbR^{m \times n}$. Computation of an upper bound for the Lipschitz constant $C$ is possible by Proposition~\ref{prop:Lipschitz:sec:D}. For practical purposes, it would also be interesting to lower bound the lower Lipschitz constant $c > 0$. While this is hard, in general, in some special cases it is possible.
\end{example}

\begin{example}[Sign retrieval]
    In sign retrieval, estimating $c$ explicitly is simplified by what is known as the $\sigma$-strong complement property \cite{balan2015invertibility, bandeira2014saving}: we are interested in the recovery of vectors $\bfx \in \bbR^n$ from magnitude-only measurements
    \begin{equation}\label{eq:phase_retrieval}
        \abs{ \langle \bfx, \bfa_i \rangle }, \qquad i \in [D],
    \end{equation}
    where $(\bfa_i)_{i = 1}^D \in \bbR^n$ is a sequence of measurement vectors. Since $\bfx$ and $-\bfx$ generate the same measurements, one typically aims to recover vectors \emph{up to a global sign}; i.e., up to the equivalence relation $\bfx \sim \bfy :\iff \bfx = \bfy \mbox{ or } \bfx = -\bfy$; or, up to action of the group $G = \lbrace -1, 1 \rbrace$ on the vector space of signals $\bbR^d$. 

    The action of the group $S_2$ by row permutation on the vector space of matrices $\bbR^{2 \times n}$ is closely related to this sign retrieval problem as demonstrated in \cite{balan2023relationships}: indeed, $\beta_\bfA : \bbR^{2 \times n} \to \bbR^{2 \times D}$, 
    \begin{equation*}
        \beta_\bfA(\bfX) := \operatorname{sort}(\bfX\bfA), \qquad \bfX \in \bbR^{2 \times n},
    \end{equation*}
    where $\bfA \in \bbR^{n \times D}$ is a matrix with column vectors $(\bfa_i)_{i = 1}^D \in \bbR^n$, separates orbits if and only if all $\bfx \in \bbR^n$ can be uniquely recovered from the measurements \eqref{eq:phase_retrieval} up to a global sign.

    However, not only the orbit separating properties of $\beta_\bfA$ are related to a sign retrieval problem, also the bi-Lipschitz condition is related to the Lipschitz properties of the sign retrieval operator $\calA : \bbR^d / \lbrace -1, 1 \rbrace \to \bbR_+^D$, 
    \begin{equation*}
        \calA(\bfx)_i := \abs{ \langle \bfx, \bfa_i \rangle }, \qquad i \in [D].
    \end{equation*}
    Indeed, the Lipschitz constant for the embedding $\beta_\bfA$ is exactly the Lipschitz constant of $\calA$ and thus given by the largest singular value $\sigma_1(\bfA)$ of the matrix $\bfA$. Moreover, the lower Lipschitz constant for the embedding $\beta_\bfA$ is also exactly the lower Lipschitz constant of $\calA$ and thus exactly given by the quantity 
    \begin{equation}\label{eq:lower_Lipschitz_PR}
        c = \min_{S \subseteq [D]} \sqrt{ \sigma_n^2(\bfA_S) + \sigma_n^2(\bfA_{S^\rmc}) }
    \end{equation}
    as demonstrated in \cite{balan2023relationships,balan2015invertibility}. Here, $\sigma_n(\cdot)$ is used to denote the $n$-th singular value (in decreasing order) of a matrix and $\bfA_S \in \bbR^{n \times \abs{S}}$ denotes the matrix obtained by only keeping the columns whose indices are elements of $S \subseteq [D]$. 

    Equation~\eqref{eq:lower_Lipschitz_PR} can allow for relatively simple computation of the lower Lipschitz constant of the permutation-invariant embedding $\beta_\bfA$. Indeed, consider $\bfA \in \bbR^{2 \times 3}$ whose columns are given by the Mercedes-Benz frame 
    \begin{equation*}
        \bfa_1 = \begin{pmatrix}
            - \sqrt{3}/2 \\
            -1/2
        \end{pmatrix}, \qquad \bfa_2 = \begin{pmatrix}
            \sqrt{3}/2 \\
            -1/2
        \end{pmatrix}, \qquad \bfa_3 = \begin{pmatrix}
            0 \\
            1
        \end{pmatrix},
    \end{equation*}
    as an example. Then, it is a simple exercise to show that $c = \tfrac{1}{\sqrt{2}}$.
\end{example}

\section{Continuous maps factor through sorting based embeddings}

In this section, we prove Theorems~\ref{thm:cor_kirszbraun} and \ref{thm:cor_dugundji}. These two results are based on two well-known extension theorems: the Kirszbraun extension theorem and the Tietze extension theorem. 

The precise version of the Kirszbraun extension theorem that we are going to use is due to F.~A.~Valentine \cite{valentine1945Lipschitz}.

\begin{theorem}[Kirszbraun--Valentine extension theorem]
    Let $H_1, H_2$ be Hilbert spaces, let $S \subseteq H_1$ and let $g_0 : S \to H_2$ be a Lipschitz continuous map with Lipschitz constant $C > 0$. Then, $g_0$ can be extended to a Lipschitz continuous map $g : H_1 \to H_2$ with the same Lipschitz constant $C$.
\end{theorem}

We note that the above theorem implies that any Lipschitz continuous map factors through any lower Lipschitz continuous map in the following sense. 

\begin{corollary}\label{cor:kirszbraun}
    Let $X$ be a metric space, let $H_1, H_2$ be Hilbert spaces, let $h : X \to H_1$ be lower Lipschitz continuous with lower Lipschitz constant $c_h > 0$ and let $f : X \to H_2$ be Lipschitz continuous with Lipschitz constant $C_f > 0$. Then, there exists $g : H_1 \to H_2$ Lipschitz continuous with Lipschitz constant at most $C_f / c_h$ such that $f = g \circ h$.
\end{corollary}

\begin{proof}
    Let $d_X$ denote the metric on $X$ and let $\norm{\cdot}_{H_i}$ denote the norm on $H_i$ where $i = 1,2$. 
    
    Since $h$ is lower Lipschitz continuous, it is injective and thus invertible on its range $\calR(h)$. We may therefore define $g_0 := f \circ h^{-1} : \calR(h) \subseteq H_1 \to H_2$. It clear that $g_0$ is Lipschitz continuous with Lipschitz constant at most $C_f/c_h$ since 
    \begin{align*}
        \norm{g_0(x) - g_0(y)}_{H_2} &= \norm{f(h^{-1}(x)) - f(h^{-1}(y))}_{H_2} \leq C_f d_X(h^{-1}(x), h^{-1}(y)) \\
        &\leq \frac{C_f}{c_h} \norm{x-y}_{H_1}
    \end{align*}
    for $x,y \in H_1$. Therefore, by the Kirszbraun--Valentine extension theorem, $g_0$ extends to $g : H_1 \to H_2$ Lipschitz continuous with Lipschitz constant at most $C_f/c_h$. Finally, we have $f = g \circ h$ by construction.
\end{proof}

\begin{remark}
    Most likely this corollary has been proven before or been included in a textbook as an exercise. We were unable to find a reference, however.
\end{remark}

We can now prove Theorem~\ref{thm:cor_kirszbraun}.

\begin{proof}[Proof of Theorem~\ref{thm:cor_kirszbraun}]
    Let $\norm{\cdot}_H$ denote the norm on $H$.
    
    Since $f$ is invariant under the action of $G$, it descends through the quotient to $\widehat f : V/{\sim} \to H$. Moreover, since $f$ is Lipschitz continuous with Lipschitz constant $C_f$, so is $\widehat f$: let $v,w \in V$ and let $g \in G$ be such that $\distqsp(v,w) = \normvsp{v - gw}$. Then, it holds that 
    \begin{equation*}
        \norm{\widehat f ([v]) - \widehat f ([w])}_H = \norm{f (v) - f (gw)}_H \leq C_f \normvsp{v-gw} = C_f \distqsp(v,w),
    \end{equation*}
    where $[v] = \set{gv}{ g \in \group}$ denotes the equivalence class of $v$ and $[w]$ denotes the equivalence class of $w$.

    If $\gamma : V \to \mathbb{R}^D$ separates orbits, then it satisfies the bi-Lipschitz condition according to our main theorem. In particular, since $\gamma$ is also invariant under the action of $G$, it also descends through the quotient to $\widehat \gamma : V/{\sim} \to H$ and $\widehat \gamma$ is lower Lipschitz continuous with lower Lipschitz constant $c > 0$.

    By Corollary~\ref{cor:kirszbraun}, there exists $g : \bbR^D \to H$ Lipschitz continuous with Lipschitz constant at most $C_f/c$ such that $\widehat f = g \circ \widehat \gamma$. Letting $v \in W$, we also get 
    \begin{equation*}
        f(v) = \widehat f([v]) = g(\widehat \gamma([v])) = g(\gamma(v)).
    \end{equation*}
\end{proof}

Similarly, the precise version of the Tietze extension theorem that we are going to use is due to J.~Dugundji \cite{dugundji1951extension}.

\begin{theorem}[Dugundji--Tietze extension theorem]
    Let $Y$ be a metric space, let $Z$ be a locally convex space, let $S \subseteq Y$ be a closed subset of $Y$ and let $g_0 : S \to Z$ be a continuous map. Then, $g_0$ can be extended to a continuous map $g : Y \to Z$ in such a way that $g(Y)$ is a subset of the convex hull of $g_0(S)$.
\end{theorem}

We note that this immediately implies that any continuous function factors through any injective, relatively open function with closed range.

\begin{corollary}\label{cor:dugundji}
    Let $X$ be a topological space, let $Y$ be a metric space and let $Z$ be a locally convex space. Let $h : X \to Y$ be an injective, relatively open map with closed range and let $f : X \to Z$ be continuous. Then, there exists $g : Y \to Z$ continuous such that $g(Y)$ is a subset of the convex hull of $f(X)$ and such that $f = g \circ h$.
\end{corollary}

\begin{proof}
    Since $h$ is injective, it is invertible on its range $\calR(h)$. We may therefore define $g_0 := f \circ h^{-1} : \calR(h) \subseteq Y \to Z$. Since $f$ is continuous and $h$ is relatively open, it follows that $g_0$ is continuous. Since the range of $h$ is closed, the Dugundji--Tietze extension theorem implies that $g_0$ extends to a continuous function $g : Y \to Z$ such that $g(Y)$ is contained in the convex hull of $g_0(\calR(h)) = f(X)$. Finally, we have $f = g \circ h$ by construction.
\end{proof}

\begin{remark}
    Again, this corollary has most likely been proven or, at least, mentioned before but we were unable to find a reference.
\end{remark}

We can now prove Theorem~\ref{thm:cor_dugundji}.

\begin{proof}[Proof of Theorem~\ref{thm:cor_dugundji}]
    Again, $f$ descends through the quotient to $\widehat f : V/{\sim} \to Z$ and, since $f$ is continuous, so is $\widehat f$. Because $\gamma$ separates orbits, our main theorem implies that $\gamma$ satisfies the bi-Lipschitz condition. Of course, $\gamma$ descends through the quotient to $\widehat \gamma : V/{\sim} \to H$ and, since $\gamma$ satisfies the bi-Lipschitz condition, $\widehat \gamma$ is bi-Lipschitz such that its range $\calR(\widehat \gamma)$ is closed and $\widehat \gamma$ is injective. Therefore, $\widehat \gamma$ is invertible on its range $\calR(\widehat \gamma)$ and its inverse is Lipschitz continuous on $\calR(\widehat \gamma)$: it follows that $\widehat \gamma$ is relatively open. By Corollary~\ref{cor:dugundji}, there exists $g : \bbR^D \to Z$ continuous such that $g(\bbR^D)$ is a subset of the convex hull of $\widehat f (V/{\sim}) = f(V)$ and $\widehat f = g \circ \widehat \gamma$. As before, we get $f = g \circ \gamma$.
\end{proof}

\section{Preliminaries for the proof of \texorpdfstring{Theorem~\ref{thm:main}}{the main theorem}}

\subsection{Sorting coorbits}

\DeclarePairedDelimiterXPP\permsort[1]{\calL}(){}{#1}
\DeclarePairedDelimiterXPP\stabsort[1]{\calH}(){}{#1}
\DeclarePairedDelimiterXPP\disthyperplane[1]{\delta}(){}{#1}
\DeclarePairedDelimiterXPP\diffminmax[1]{\Delta}(){}{#1}

Denote the set of permutations that sort a vector $\bfx \in \bbR^\ordgroup$ by 
\begin{equation*}
    \permsort{\bfx} := \set{ \sigma \in S_\ordgroup }{ \sigma \bfx = \sort{ \bfx } }.
\end{equation*}
Let us furthermore denote the stabiliser of $\bfx$ by 
\begin{equation*}
    \stabsort{\bfx} := \set{ \sigma \in S_\ordgroup }{ \sigma \bfx = \bfx }.
\end{equation*}
The sets $\permsort{\bfx}$ and $\stabsort{\bfx}$ are related in the following way.

\begin{proposition}\label{prop:sorting_stabiliser}
    Let $\bfx \in \bbR^\ordgroup$ and $\sigma \in \permsort{\bfx}$. Then, $\permsort{\bfx} = \sigma \stabsort{\bfx}$.
\end{proposition}


Next, we may denote 
\begin{equation*}
    \disthyperplane{\bfx} := \min_{\substack{i,j \in [\ordgroup] \\ x_i \neq x_j}} \abs{x_i - x_j}.
\end{equation*}
If the vector $\bfx$ is \emph{not} constant, then $\disthyperplane{\bfx}$ is exactly a scalar multiple of the distance of $\bfx$ to the next hyperplane of the form $\set{ \bfx \in \bbR^\ordgroup }{ x_i = x_j }$, for $i,j \in [\ordgroup]$ with $i < j$ that does \emph{not} contain $\bfx$. 

\begin{proposition}\label{prop:reexpress_d}
    Let $\bfx \in \bbR^\ordgroup$. Then, 
    \begin{equation*}
        \disthyperplane{\bfx} = \min_{\sigma \in S_\ordgroup \setminus \stabsort{\bfx}} \norm{(\sigma-e) \bfx}_\infty = \min_{\sigma \in S_\ordgroup \setminus \permsort{\bfx}} \norm{\sigma \bfx - \sort{\bfx}}_\infty.
    \end{equation*}
\end{proposition}

Finally, let us denote the difference of the maximum and minimum of a vector $\bfx$ by 
\begin{equation*}
    \diffminmax{\bfx} := \max_{i \in [\ordgroup]} x_i - \min_{i \in [\ordgroup]} x_i = \max_{i,j \in [\ordgroup]} \abs{x_i - x_j}.
\end{equation*}
Introducing the matrix $\bfD \in \bbR^{\ordgroup(\ordgroup-1)/2 \times \ordgroup}$,
\begin{equation*}
    \bfD = \begin{pNiceMatrix}
        1       & -1        &           &           &       \\
        \Vdots  &           & -1        &           &       \\
        \Vdots  &           &           & \Ddots    &       \\
        1       &           &           &           & -1    \\
        0       & 1         & -1        &           &       \\
        \Vdots  & \Vdots    &           & \Ddots    &       \\
        0       & 1         &           &           & -1    \\
                &           & \vdots    &           &       \\
        0       & \Cdots    & 0         & 1         & -1
    \CodeAfter
    \SubMatrix.{1-1}{4-5}\}[right-xshift=1em,name=a]
    \tikz \node [right] at (a-right.east) {$\ordgroup-1$};
    \SubMatrix.{5-1}{7-5}\}[right-xshift=1em,name=b]
    \tikz \node [right] at (b-right.east) {$\ordgroup-2$};
    \SubMatrix.{9-1}{9-5}\}[right-xshift=1em,name=c]
    \tikz \node [right] at (c-right.east) {$1$};
    \end{pNiceMatrix},
\end{equation*}
we may write $\diffminmax{\bfx} = \norm{\bfD \bfx}_\infty$, which shows that $\diffminmax{\cdot}$ satisfies the triangle inequality and is absolutely homogeneous. Finally, we note the following simple inequality.

\begin{proposition}\label{prop:useless_inequality}
    Let $\bfx \in \bbR^\ordgroup$. Then, $\disthyperplane{\bfx} \leq \diffminmax{\bfx}$.
\end{proposition}


Inspired by \cite{balan2023GI}, we show a score of set inclusions and equalities. Let us start by proving the following simple inclusions (cf.~\cite[Lemma~2.5 on p.~7]{balan2023GI}).

\begin{lemma}\label{lem:simple_inclusion:sec:D}
    Let $\bfx,\bfy \in \bbR^\ordgroup$ be such that $\diffminmax{\bfy} < \disthyperplane{\bfx}$. Then, $\permsort{\bfx + \bfy} \subseteq \permsort{\bfx}$ and $\stabsort{\bfx + \bfy} \subseteq \stabsort{\bfx}$.
\end{lemma}

\begin{proof}
    We show that $\permsort{\bfx}^\rmc \subseteq \permsort{\bfx + \bfy}^\rmc$: if $\sigma \not\in \permsort{\bfx}$, then we may find $i,j \in [\ordgroup]$ such that $i < j$ and $x_{\sigma(i)} < x_{\sigma(j)}$. Therefore, 
    \begin{align*}
        x_{\sigma(j)} + y_{\sigma(j)} - x_{\sigma(i)} - y_{\sigma(i)} &\geq \min_{\substack{i,j \in [\ordgroup] \\ x_i \neq x_j}} \abs{x_i - x_j} + \min_{i \in [\ordgroup]} y_i - \max_{i \in [\ordgroup]} y_i \\
        &= \disthyperplane{\bfx} - \diffminmax{\bfy} > 0
    \end{align*}
    such that $\sigma \not\in \permsort{\bfx + \bfy}$. Finally, pick $\sigma \in \permsort{\bfx + \bfy} \subseteq \permsort{\bfx}$ and note that $\stabsort{\bfx + \bfy} = \sigma^{-1} \permsort{\bfx + \bfy} \subseteq \sigma^{-1} \permsort{\bfx} = \stabsort{\bfx}$ according to Proposition~\ref{prop:sorting_stabiliser}.
\end{proof}

Using this lemma, we can show that the set of sorting permutations does not change along certain straight line segments (cf.~\cite[Lemma~2.6 on p.~8]{balan2023GI}).

\begin{lemma}\label{lem:equality_along_interval:sec:D}
    Let $\bfx, \bfy \in \bbR^\ordgroup$ be such that $\diffminmax{\bfy} < \disthyperplane{\bfx}$. Then, $\permsort{ \bfx + t \bfy } = \permsort{\bfx + \bfy}$ and $\stabsort{ \bfx + t \bfy } = \stabsort{\bfx + \bfy}$ for $t \in (0,1]$. 
\end{lemma}

\begin{proof}
    \begin{proof}[$\subseteq$]\let\qed\relax
        Let $\sigma \in \permsort{ \bfx + t \bfy }$. Then, Lemma~\ref{lem:simple_inclusion:sec:D} implies that $\sigma \in \permsort{\bfx}$. So, let $i,j \in [\ordgroup]$ be such that $i < j$. There are two cases: either $x_{\sigma(i)} = x_{\sigma(j)}$ in which case 
        \begin{align*}
            (\bfx + \bfy)_{\sigma(i)} - (\bfx + \bfy)_{\sigma(j)} &= y_{\sigma(i)} - y_{\sigma(j)} = t^{-1} (ty_{\sigma(i)} - ty_{\sigma(j)}) \\
            &= t^{-1} \left( (\bfx + t\bfy)_{\sigma(i)} - (\bfx + t\bfy)_{\sigma(j)} \right) \geq 0;
        \end{align*}
        or $x_{\sigma(i)} > x_{\sigma(j)}$ in which case 
        \begin{align*}
            (\bfx + \bfy)_{\sigma(i)} - (\bfx + \bfy)_{\sigma(j)} &= x_{\sigma(i)} - x_{\sigma(j)} + y_{\sigma(i)} - y_{\sigma(j)} \geq \disthyperplane{\bfx} - \diffminmax{\bfy} > 0.
        \end{align*}
        Either way, $(\bfx + \bfy)_{\sigma(i)} \geq (\bfx + \bfy)_{\sigma(j)}$. Since $i,j \in [\ordgroup]$ were arbitrary, we conclude that $\sigma \in \permsort{\bfx + \bfy}$.
    \end{proof}

    \begin{proof}[$\supseteq$]\let\qed\relax
        The reverse inclusion follows because $\bfx + t \bfy = t (\bfx + \bfy) + (1-t) \bfy$ is a convex combination of $\bfx + \bfy$ and $\bfx$. More precisely, we have that $\sigma \in \permsort{\bfx + \bfy} \subseteq \permsort{\bfx}$ such that 
        \begin{align*}
            (\bfx+t\bfy)_{\sigma(i)} &= t (\bfx+\bfy)_{\sigma(i)} + (1-t) x_{\sigma(i)} \\
            &\geq t (\bfx+\bfy)_{\sigma(j)} + (1-t) x_{\sigma(j)} = (\bfx+t\bfy)_{\sigma(j)}
        \end{align*}
        for all $i,j \in [\ordgroup]$ such that $i < j$. Therefore, $\sigma \in \permsort{ \bfx + t \bfy }$.
    \end{proof}

    The equality for the stabilisers follows from the equality for the permutations that sort by Proposition~\ref{prop:sorting_stabiliser}.
\end{proof}

Next, we prove that the set of permutations that sort is stable on sufficiently small hypercubes (cf.~\cite[Lemma~2.8, items~1 and 2 on p.~10]{balan2023GI}).

\begin{lemma}\label{lem:stability_of_sets:sec:D}
    Let $p \in \bbN$, $(\bfx_k)_{k = 1}^p \in \bbR^\ordgroup$, $(c_k)_{k = 1}^p \in \bbR_+$ and $\epsilon \in (0,1)$ be such that 
    \begin{gather*}
        \diffminmax{ \bfx_{\ell+1} } < \disthyperplane[\Bigg]{ \sum_{k = 1}^\ell \bfx_k }, \qquad \ell \in [p-1], \\
        \diffminmax[\Bigg]{ \sum_{k = 1}^p (c_k - 1) \bfx_k } \leq \epsilon \cdot \disthyperplane[\Bigg]{ \sum_{k = 1}^p \bfx_k }.
    \end{gather*}
    Then, it holds that 
    \begin{gather*}
        \text{1.~}~\permsort[\Bigg]{ \sum_{k = 1}^p \bfx_k } = \permsort[\Bigg]{ \sum_{k = 1}^p c_k \bfx_k }, \qquad\text{2.~}~\stabsort[\Bigg]{ \sum_{k = 1}^p \bfx_k } = \stabsort[\Bigg]{ \sum_{k = 1}^p c_k \bfx_k }, \\ 
        \text{3.~}~(1-\epsilon) \cdot \disthyperplane[\Bigg]{ \sum_{k = 1}^p \bfx_k } \leq \disthyperplane[\Bigg]{ \sum_{k = 1}^p c_k \bfx_k } \leq (1+\epsilon) \cdot \disthyperplane[\Bigg]{ \sum_{k = 1}^p \bfx_k }.
    \end{gather*}
\end{lemma}

\begin{proof}
    \begin{proof}[$\supseteq$]\let\qed\relax
        We note that 
        \begin{equation*}
            \sum_{k = 1}^p c_k \bfx_k = \sum_{k = 1}^p \bfx_k + \sum_{k = 1}^p (c_k-1) \bfx_k, \qquad \diffminmax[\Bigg]{ \sum_{k = 1}^p (c_k - 1) \bfx_k } < \disthyperplane[\Bigg]{ \sum_{k = 1}^p \bfx_k }.
        \end{equation*}
        Therefore, 
        \begin{equation*}
            \permsort[\Bigg]{ \sum_{k = 1}^p c_k \bfx_k } \subseteq \permsort[\Bigg]{ \sum_{k = 1}^p \bfx_k }
        \end{equation*}
        follows from Lemma~\ref{lem:simple_inclusion:sec:D}. We conclude that the stabilisers satisfy the same inclusion.
    \end{proof}

    \begin{proof}[$\subseteq$]\let\qed\relax
        Applying Lemma~\ref{lem:simple_inclusion:sec:D} inductively yields 
        \begin{equation*}
            \permsort[\Bigg]{ \sum_{k = 1}^p \bfx_k } \subseteq \permsort[\Bigg]{ \sum_{k = 1}^{p-1} \bfx_k } \subseteq \dots \subseteq \permsort{ \bfx_1 }.
        \end{equation*}
        So, if $\sigma \in \permsort{ \sum_{k = 1}^p \bfx_k }$ and $\tau \in \permsort{ \sum_{k = 1}^p c_k \bfx_k } \subseteq \permsort{ \sum_{k = 1}^p \bfx_k }$, then 
        \begin{gather*}
            \sigma \bfx_1 = \sort{ \bfx_1 } = \tau \bfx_1, \\
            \sigma (\bfx_1 + \bfx_2) = \sort{ \bfx_1 + \bfx_2 } = \tau (\bfx_1 + \bfx_2), \\
            \vdots \\
            \sigma \Bigg( \sum_{k = 1}^p \bfx_k \Bigg) = \sort[\Bigg]{ \sum_{k = 1}^p \bfx_k } = \tau \Bigg( \sum_{k = 1}^p \bfx_k \Bigg),
        \end{gather*}
        such that $\sigma \bfx_k = \tau \bfx_k$ for $k \in [p]$. Therefore, 
        \begin{equation*}
            \sort[\Bigg]{ \sum_{k = 1}^p c_k \bfx_k } = \tau\Bigg( \sum_{k = 1}^p c_k \bfx_k \Bigg) = \sum_{k = 1}^p c_k \tau \bfx_k = \sum_{k = 1}^p c_k \sigma \bfx_k = \sigma\Bigg( \sum_{k = 1}^p c_k \bfx_k \Bigg)
        \end{equation*}
        and $\sigma \in \permsort{ \sum_{k = 1}^p c_k \bfx_k }$. It follows that the stabilisers satisfy the same inclusion.
    \end{proof}

    \begin{proof}[Inequalities]\let\qed\relax
        Before proving the two inequalities, we establish the following claim.
    \end{proof}

    \begin{proof}[Claim]\let\qed\relax
        Let $i,j \in [\ordgroup]$. Then, 
        \begin{equation*}
            \Bigg( \sum_{k = 1}^p c_k \bfx_k \Bigg)_i > \Bigg( \sum_{k = 1}^p c_k \bfx_k \Bigg)_j \iff \Bigg( \sum_{k = 1}^p \bfx_k \Bigg)_i > \Bigg( \sum_{k = 1}^p \bfx_k \Bigg)_j.
        \end{equation*}
    \end{proof}

    \begin{proof}[Proof of the claim]\let\qed\relax
        First, note that $(ij) \in \stabsort{ \sum_{k = 1}^p c_k \bfx_k }$ if and only if $(ij) \in \stabsort{ \sum_{k = 1}^p \bfx_k }$ according to item~2. Now, assume by contradiction that 
        \begin{equation*}
            \Bigg( \sum_{k = 1}^p c_k \bfx_k \Bigg)_i > \Bigg( \sum_{k = 1}^p c_k \bfx_k \Bigg)_j, \qquad \Bigg( \sum_{k = 1}^p \bfx_k \Bigg)_i < \Bigg( \sum_{k = 1}^p \bfx_k \Bigg)_j
        \end{equation*}
        and consider the function $f : [0,1] \to \bbR$,
        \begin{equation*}
            f(t) := \sum_{k = 1}^p (1 + t(c_k-1)) \left( (\bfx_k)_i - (\bfx_k)_j \right),
        \end{equation*}
        which satisfies $f(0) < 0$ and $f(1) > 0$. By the intermediate value theorem, there exists $t \in (0,1)$ such that $f(t) = 0$. By item~2 and 
        \begin{equation*}
            \diffminmax[\Bigg]{ \sum_{k = 1}^p ((1 + t(c_k-1))-1) \bfx_k } = t \cdot \diffminmax[\Bigg]{ \sum_{k = 1}^p (c_k-1) \bfx_k } < t \cdot \disthyperplane[\Bigg]{ \sum_{k = 1}^p \bfx_k },
        \end{equation*}
        we may conclude that 
        \begin{equation*}
            (ij) \in \stabsort[\Bigg]{ \sum_{k = 1}^p (1 + t(c_k-1)) \bfx_k } = \stabsort[\Bigg]{ \sum_{k = 1}^p \bfx_k }
        \end{equation*}
        which contradicts the assumption $( \sum_{k = 1}^p \bfx_k )_i < ( \sum_{k = 1}^p \bfx_k )_j$.
        The reverse implication follows by noting that
        \begin{equation*}
            (ij) \in \stabsort[\Bigg]{ \sum_{k = 1}^p \bfx_k } = \stabsort[\Bigg]{ \sum_{k = 1}^p c_k \bfx_k }
        \end{equation*}
        contradicts the assumption $( \sum_{k = 1}^p c_k \bfx_k )_i > ( \sum_{k = 1}^p c_k \bfx_k )_j$ and exchanging the role of $i$ and $j$.
    \end{proof}

    Now, let $i,j \in [\ordgroup]$ be such that $( \sum_{k = 1}^p c_k \bfx_k )_i > ( \sum_{k = 1}^p c_k \bfx_k )_j$ and 
    \begin{equation*}
        \disthyperplane[\Bigg]{ \sum_{k = 1}^p c_k \bfx_k } = \Bigg( \sum_{k = 1}^p c_k \bfx_k \Bigg)_i - \Bigg( \sum_{k = 1}^p c_k \bfx_k \Bigg)_j > 0.
    \end{equation*}
    Then, we may use the claim to see that 
    \begin{align*}
        \disthyperplane[\Bigg]{ \sum_{k = 1}^p c_k \bfx_k } ={}& \Bigg( \sum_{k = 1}^p \bfx_k \Bigg)_i - \Bigg( \sum_{k = 1}^p \bfx_k \Bigg)_j \\
        & \qquad + \Bigg( \sum_{k = 1}^p (c_k-1) \bfx_k \Bigg)_i - \Bigg( \sum_{k = 1}^p (c_k-1) \bfx_k \Bigg)_j \\
        \geq{}& \disthyperplane[\Bigg]{ \sum_{k = 1}^p \bfx_k } - \diffminmax[\Bigg]{ \sum_{k = 1}^p (c_k-1) \bfx_k } \geq (1-\epsilon) \cdot \disthyperplane[\Bigg]{ \sum_{k = 1}^p \bfx_k }.
    \end{align*}
    Vice versa, let $i,j \in [\ordgroup]$ be such that $( \sum_{k = 1}^p \bfx_k )_i > ( \sum_{k = 1}^p \bfx_k )_j$ and 
    \begin{equation*}
        \disthyperplane[\Bigg]{ \sum_{k = 1}^p \bfx_k } = \Bigg( \sum_{k = 1}^p \bfx_k \Bigg)_i - \Bigg( \sum_{k = 1}^p \bfx_k \Bigg)_j > 0.
    \end{equation*}
    Using the claim once more yields
    \begin{align*}
        \disthyperplane[\Bigg]{ \sum_{k = 1}^p \bfx_k } ={}& \Bigg( \sum_{k = 1}^p c_k \bfx_k \Bigg)_i - \Bigg( \sum_{k = 1}^p c_k \bfx_k \Bigg)_j \\
        & \qquad + \Bigg( \sum_{k = 1}^p (1-c_k) \bfx_k \Bigg)_i - \Bigg( \sum_{k = 1}^p (1-c_k) \bfx_k \Bigg)_j \\
        \geq{}& \disthyperplane[\Bigg]{ \sum_{k = 1}^p c_k \bfx_k } - \diffminmax[\Bigg]{ \sum_{k = 1}^p (c_k-1) \bfx_k } \\
        \geq{}& \disthyperplane[\Bigg]{ \sum_{k = 1}^p c_k \bfx_k } - \epsilon \cdot \disthyperplane[\Bigg]{ \sum_{k = 1}^p \bfx_k }
    \end{align*}
    which proves the upper bound after rearrangement of the inequality.
\end{proof}

Finally, we show that the set of permutations that sort remains stable on sufficiently small hypercubes when we add a sufficiently small additional vector (cf.~\cite[Lemma~2.8, item~3 on p.~10]{balan2023GI}).

\begin{lemma}\label{lem:finally_done:sec:D}
    Let $p \in \bbN$, $(\bfx_k)_{k = 1}^p,\bfy \in \bbR^\ordgroup$ and $(c_k)_{k = 1}^p \in \bbR_+$ be such that 
    \begin{gather*}
        \diffminmax{\bfx_{\ell+1}} < \disthyperplane[\Bigg]{ \sum_{k = 1}^\ell \bfx_k }, \qquad \ell \in [p-1], \\
        \diffminmax[\Bigg]{ \sum_{k = 1}^p (c_k - 1) \bfx_k } \leq \frac12 \cdot \disthyperplane[\Bigg]{ \sum_{k = 1}^p \bfx_k }, \\
        \diffminmax{\bfy} \leq \frac12 \cdot \disthyperplane[\Bigg]{ \sum_{k = 1}^p \bfx_k }.
    \end{gather*}
    Then,
    \begin{gather*}
        \permsort[\Bigg]{ \sum_{k = 1}^p \bfx_k + \bfy } = \permsort[\Bigg]{ \sum_{k = 1}^p c_k \bfx_k + \bfy }, \\
        \stabsort[\Bigg]{ \sum_{k = 1}^p \bfx_k + \bfy } = \stabsort[\Bigg]{ \sum_{k = 1}^p c_k \bfx_k + \bfy }.
    \end{gather*}
\end{lemma}

\begin{proof}
    Before proving the two inclusions, we show the following claim.

    \begin{proof}[Claim]\let\qed\relax
        Let $i,j \in [\ordgroup]$. Then, 
        \begin{equation*}
            \Bigg( \sum_{k = 1}^p \bfx_k \Bigg)_i = \Bigg( \sum_{k = 1}^p \bfx_k \Bigg)_j \implies \forall k \in [p]: ( \bfx_k )_i = ( \bfx_k )_j.
        \end{equation*}
    \end{proof}

    \begin{proof}[Proof of the claim]\let\qed\relax
        If $( \sum_{k = 1}^p \bfx_k )_i = ( \sum_{k = 1}^p \bfx_k )_j$, then $(ij) \in \stabsort{ \sum_{k = 1}^p \bfx_k }$. According to Lemma~\ref{lem:simple_inclusion:sec:D}, we have 
        \begin{equation*}
            (ij) \in \stabsort[\Bigg]{ \sum_{k = 1}^p \bfx_k } \subseteq \stabsort[\Bigg]{ \sum_{k = 1}^{p-1} \bfx_k } \subseteq \dots \subseteq \stabsort{ \bfx_1 }.
        \end{equation*}
        Therefore, 
        \begin{gather*}
            ( \bfx_1 )_i = ( \bfx_1 )_j, \\
            ( \bfx_1 + \bfx_2 )_i = ( \bfx_1 + \bfx_2 )_j, \\
            \vdots \\
            \Bigg( \sum_{k = 1}^p \bfx_k \Bigg)_i = \Bigg( \sum_{k = 1}^p \bfx_k \Bigg)_j,
        \end{gather*}
        which implies $( \bfx_k )_i = ( \bfx_k )_j$ for $k \in [p]$. 
    \end{proof}

    \begin{proof}[$\subseteq$]\let\qed\relax
        Let $\sigma \in \permsort{ \sum_{k = 1}^p \bfx_k + \bfy } \subseteq \permsort{ \sum_{k = 1}^p \bfx_k }$ (according to Lemma~\ref{lem:simple_inclusion:sec:D}) and let $i,j \in [\ordgroup]$ be such that $i < j$.
        
        We may now consider two cases: if $\sum_{k = 1}^p (\bfx_k)_{\sigma(i)} = \sum_{k = 1}^p (\bfx_k)_{\sigma(j)}$, then 
        \begin{multline*}
            \Bigg( \sum_{k = 1}^p c_k \bfx_k + \bfy \Bigg)_{\sigma(i)} - \Bigg( \sum_{k = 1}^p c_k \bfx_k + \bfy \Bigg)_{\sigma(j)} = y_{\sigma(i)} - y_{\sigma(j)} \\
            = \Bigg( \sum_{k = 1}^p \bfx_k + \bfy \Bigg)_{\sigma(i)} - \Bigg( \sum_{k = 1}^p \bfx_k + \bfy \Bigg)_{\sigma(j)}\geq 0
        \end{multline*}
        according to the claim.
        
        If, vice versa, $\sum_{k = 1}^p (\bfx_k)_{\sigma(i)} > \sum_{k = 1}^p (\bfx_k)_{\sigma(j)}$, then
        \begin{align*}
            \MoveEqLeft[3] \Bigg( \sum_{k = 1}^p c_k \bfx_k + \bfy \Bigg)_{\sigma(i)} - \Bigg( \sum_{k = 1}^p c_k \bfx_k + \bfy \Bigg)_{\sigma(j)} \\
            &= \Bigg( \sum_{k = 1}^p c_k \bfx_k \Bigg)_{\sigma(i)} - \Bigg( \sum_{k = 1}^p c_k \bfx_k \Bigg)_{\sigma(j)} + y_{\sigma(i)} - y_{\sigma(j)} \\
            &\geq \disthyperplane[\Bigg]{ \sum_{k = 1}^p c_k \bfx_k } - \diffminmax{\bfy} \geq \frac12 \cdot \disthyperplane[\Bigg]{ \sum_{k = 1}^p \bfx_k } - \diffminmax{\bfy} \geq 0,
        \end{align*}
        where we use the claim in the proof of Lemma~\ref{lem:stability_of_sets:sec:D} and Lemma~\ref{lem:stability_of_sets:sec:D} itself. So, in both cases $( \sum_{k = 1}^p c_k \bfx_k + \bfy )_{\sigma(i)} \geq ( \sum_{k = 1}^p c_k \bfx_k + \bfy )_{\sigma(j)}$. Since $i,j \in [\ordgroup]$ are arbitrary indices satisfying $i < j$, we can conclude that $\sigma \in \permsort{ \sum_{k = 1}^p c_k \bfx_k + \bfy }$.
    \end{proof}

    \begin{proof}[$\supseteq$]\let\qed\relax
        Let $\sigma \in \permsort{ \sum_{k = 1}^p c_k \bfx_k + \bfy } \subseteq \permsort{ \sum_{k = 1}^p c_k \bfx_k } = \permsort{ \sum_{k = 1}^p \bfx_k }$ and $\tau \in \permsort{ \sum_{k = 1}^p \bfx_k + \bfy } \subseteq \permsort{ \sum_{k = 1}^p c_k \bfx_k + \bfy } \subseteq \permsort{ \sum_{k = 1}^p \bfx_k }$. Then, we can show that $\sigma \bfx_k = \tau \bfx_k$ for $k \in [p]$, as in the proof of Lemma~\ref{lem:stability_of_sets:sec:D}. Therefore, 
        \begin{equation*}
            \sigma \Bigg( \sum_{k = 1}^p c_k \bfx_k + \bfy \Bigg) = \sort[\Bigg]{ \sum_{k = 1}^p c_k \bfx_k + \bfy } = \tau \Bigg( \sum_{k = 1}^p c_k \bfx_k + \bfy \Bigg)
        \end{equation*}
        implies $\sigma \bfy = \tau \bfy$. Finally, we have 
        \begin{align*}
            \sort[\Bigg]{ \sum_{k = 1}^p \bfx_k + \bfy } &= \tau \Bigg( \sum_{k = 1}^p \bfx_k + \bfy \Bigg) = \sum_{k = 1}^p \tau \bfx_k + \tau \bfy = \sum_{k = 1}^p \sigma \bfx_k + \sigma \bfy \\
            &= \sigma \Bigg( \sum_{k = 1}^p \bfx_k + \bfy \Bigg)
        \end{align*}
        such that $\sigma \in \permsort{ \sum_{k = 1}^p \bfx_k + \bfy }$.
    \end{proof}

    The equality for the stabilisers follows immediately.
\end{proof}

\subsection{Stabilisers of the group action}

\DeclarePairedDelimiterXPP\stabgroup[1]{H}(){}{#1}

The stabilisers $\stabgroup{v} := \set{ g \in \group }{ gv = v }$ of vectors $v \in \vsp$ under the group action satisfy an inclusion similar to the one presented in Lemma~\ref{lem:simple_inclusion:sec:D}

\begin{lemma}\label{lem:simple_inclusion_stabilisers}
    Let $v,w \in \vsp$ be such that 
    \begin{equation*}
        \normvsp{w} < \frac12 \cdot \min_{g \not\in \stabgroup{v}} \normvsp{(e-g) v}.
    \end{equation*}
    Then, $\stabgroup{v + w} \subseteq \stabgroup{v}$.
\end{lemma}

\begin{proof}
    Let $g \not\in \stabgroup{v}$ and consider 
    \begin{align*}
        \normvsp{g(v+w) - (v+w)} &\geq \normvsp{(e-g)v} - 2 \normvsp{w} \\
        &\geq \min_{g \not\in \stabgroup{v}} \normvsp{(e-g) v} - 2 \normvsp{w} > 0
    \end{align*}
    such that $g \not\in \stabgroup{v+w}$.
\end{proof}

\section{The proof of \texorpdfstring{Theorem~\ref{thm:main}}{the main theorem}}

Throughout this section, we will work with the map $\widehat \gamma : \vsp/{\sim} \to \bbR^D$ naturally obtained from setting $\widehat \gamma([v]) = \gamma(v)$ for $v \in \vsp$. Here, $[v] = \set{gv}{g\in\group}$ denotes the equivalence class of $v$. Throughout the rest of this section, we will drop the brackets and denote the equivalence class by $v$ as well.

We will start by proving that $\widehat \embedding$ is Lipschitz continuous.

\newcommand{\collectioncoorbits}{K}

\begin{proposition}[Lipschitz continuity]\label{prop:Lipschitz:sec:D}
    Let $\collectioncoorbits : \vsp \to \bbR^{\ordgroup \nwin}$ denote the collection of coorbits 
    \begin{equation*}
        \collectioncoorbits v = \begin{pmatrix}
            \coorbit{\phi_1} v \\ \vdots \\ \coorbit{\phi_\nwin} v
        \end{pmatrix}, \qquad v \in \vsp,
    \end{equation*}
    and let $\norm{\collectioncoorbits}_\mathrm{op} := \max_{\normvsp{v}=1} \norm{Kv}_2$ denote its operator norm. Then,
    \[
        \widehat \embedding : (\vsp / {\sim},\distqsp) \to (\bbR^\embeddingdim,\norm{\cdot}_2)
    \]
    is Lipschitz continuous with Lipschitz constant bounded by $\norm{\alpha}_{\mathrm{F} \to 2} \cdot \norm{\collectioncoorbits}_\mathrm{op}$.
\end{proposition}

\begin{proof}
    Somewhat similar to the proofs of \cite[Theorem~3.9 on p.~15]{balan2022permutation} and \cite[Lemma~2.3 on p.~5]{balan2023GI}, we let $v,w \in \vsp$ be arbitrary and estimate 
    \begin{align*}
        \MoveEqLeft[3] \norm{ \embedding(v) - \embedding(w) }_2^2 \\
        ={}& \norm*{ \alpha \left( \begin{pmatrix}
            \sort{ \kappa_{\phi_1} v } - \sort{ \kappa_{\phi_1} w } & \dots & \sort{ \kappa_{\phi_\nwin} v } - \sort{ \kappa_{\phi_\nwin} v }
        \end{pmatrix} \right) }_2^2 \\
        \leq{}& \norm{\alpha}_{\mathrm{F} \to 2}^2 \cdot \norm*{ \begin{pmatrix}
            \sort{ \kappa_{\phi_1} v } - \sort{ \kappa_{\phi_1} w } & \dots & \sort{ \kappa_{\phi_\nwin} v } - \sort{ \kappa_{\phi_\nwin} v }
        \end{pmatrix} }_\mathrm{F}^2 \\
        ={}& \norm{\alpha}_{\mathrm{F} \to 2}^2 \cdot \sum_{\ell = 1}^\nwin \norm{ \sort{ \kappa_{\phi_\ell} v } - \sort{ \kappa_{\phi_\ell} w } }_2^2 \leq \norm{\alpha}_{\mathrm{F} \to 2}^2 \cdot \sum_{\ell = 1}^\nwin \norm{ \kappa_{\phi_\ell} (v - w) }_2^2 \\
        ={}& \norm{\alpha}_{\mathrm{F} \to 2}^2 \cdot \norm{K(v-w)}_2^2 \leq \norm{\alpha}_{\mathrm{F} \to 2}^2 \cdot \norm{K}_\mathrm{op}^2 \cdot \normvsp{ v-w }^2.
    \end{align*}
    Now, let $v,w \in \vsp$ and let $g \in \group$ be such that $\distqsp(v,w) = \normvsp{v - g w}$. Then, 
    \begin{align*}
        \norm{ \widehat \embedding(v) - \widehat \embedding(w) }_2 &= \norm{ \embedding(v) - \embedding(gw) }_2 \leq \norm{\alpha}_{\mathrm{F} \to 2} \cdot \norm{K}_\mathrm{op} \cdot \normvsp{v - gw} \\
        &= \norm{\alpha}_{\mathrm{F} \to 2} \cdot \norm{K}_\mathrm{op} \cdot \distqsp(v,w)
    \end{align*}
    as desired.
\end{proof}

\begin{remark}
    If we choose an orthonormal basis for $\vsp$ and express $\collectioncoorbits$ as a matrix with respect to that basis, then $\norm{K}_\mathrm{op}$ coincides with the largest singular value of that matrix.
\end{remark}

The inductive procedure for the rest of the proof of the main theorem (Theorem~\ref{thm:main}) is adapted from \cite{balan2023GI}. The argument for the base case is also contained in \cite{balan2022permutation}.

\begin{lemma}[Base case]\label{lem:base_case:sec:D}
    If $\widehat \embedding : (\vsp/{\sim},\distqsp) \to (\bbR^D,\norm{\cdot}_2)$ is injective and \emph{not} lower Lipschitz continuous, then there exists $z_1 \in \vsp$, $\normvsp{z_1} = 1$, at which the local lower Lipschitz constant of $\widehat \embedding$ vanishes; more precisely, there exist sequences $(v_i)_{i = 1}^\infty$, $(w_i)_{i = 1}^\infty \in \vsp$ such that $v_i \not\sim w_i$ for $i \in \bbN$, $\lim_{i \to \infty} v_i = \lim_{i \to \infty} w_i = z_1$ and 
    \begin{equation*}
        \lim_{i \to \infty} \frac{\norm{ \embedding(v_i) - \embedding(w_i) }_2}{\distqsp(v_i,w_i)} = 0.
    \end{equation*}
\end{lemma}

\begin{proof}
    Since we assume that $\widehat \embedding$ is not lower Lipschitz continuous, there exist sequences $(v_i)_{i = 1}^\infty, (w_i)_{i = 1}^\infty \in \vsp$ such that $v_i \not\sim w_i$ for $i \in \bbN$ and 
    \begin{equation*}
        \lim_{i \to \infty} \frac{\norm{ \embedding(v_i) - \embedding(w_i) }_2}{\distqsp(v_i,w_i)} = 0.
    \end{equation*}
    The fraction is invariant under the transformation $(v,w) \mapsto r(w,v)$, $r > 0$, such that we can assume without loss of generality that $\normvsp{v_i} \leq \normvsp{w_i} = 1$. Since the unit ball in finite-dimensional vector spaces is compact, we can extract subsequences along which both $(v_i)_{i = 1}^\infty$ and $(w_i)_{i = 1}^\infty$ converge. Let us pass to these subsequences and define 
    \begin{equation*}
        f_1 := v_\infty := \lim_{i \to \infty} v_i, \qquad w_\infty := \lim_{i \to \infty} w_i.
    \end{equation*}
    We may now use the continuity of $\embedding$ (which follows because $\widehat \embedding$ and thus $\embedding$ is Lipschitz continuous) to see that 
    \begin{equation*}
        \norm{ \embedding(v_\infty) - \embedding(w_\infty) }_2 = \lim_{i \to \infty} \norm{ \embedding(v_i) - \embedding(w_i) }_2 \leq 2 \lim_{i \to \infty} \frac{\norm{ \embedding(v_i) - \embedding(w_i) }_2}{\distqsp(v_i,w_i)} = 0.
    \end{equation*}
    Therefore, $\embedding(v_\infty) = \embedding(w_\infty)$ and the injectivity of $\widehat \embedding$ implies that $f_1 = v_\infty = g w_\infty$ for some $g \in \group$. It follows immediately that $\normvsp{f_1} = \normvsp{w_\infty} = 1$. We finally define $w_i' := g w_i$ and note that $v_i \not\sim w_i'$ for $i \in \bbN$, $\lim_{i \to \infty} w_i' = f_1$ and that 
    \begin{equation*}
        \lim_{i \to \infty} \frac{\norm{ \embedding(v_i) - \embedding(w_i') }_2}{\distqsp(v_i,w_i')} = \lim_{i \to \infty} \frac{\norm{ \embedding(v_i) - \embedding(w_i) }_2}{\distqsp(v_i,w_i)} = 0.
    \end{equation*}
\end{proof}

Before proceeding with the induction step, we prove the following lemma which is similar to the first claim in the proof of \cite[Lemma~2.13 on p.~19]{balan2023GI}. It will be used in the proof of the induction step and the final theorem.

\begin{lemma}\label{lem:proofofmaintheorem}
    Let $p \in \bbN$, $p \leq d$, let $(f_j)_{j = 1}^p \in \vsp \setminus \{0\}$ be an orthogonal sequence such that
    \begin{equation}\label{eq:matryoshka_doll}
        \forall \ell \in [\nwin], k \in [p-1] : \diffminmax{ \coorbit{\phi_\ell} f_{k+1} } < \disthyperplane[\Bigg]{ \sum_{j=1}^k \coorbit{\phi_\ell} f_j }.
    \end{equation}
    Let $F := \operatorname{sp}(f_j)_{j = 1}^p \subseteq \vsp$ and assume that the local lower Lipschitz constant of
    \begin{equation*}
        \widehat \embedding|_{F} : (F/{\sim},\distqsp) \to (\bbR^\embeddingdim,\norm{\cdot}_2)
    \end{equation*}
    vanishes at $f := f_1 + f_2 + \dots + f_p$. Then, $\widehat \embedding : (\vsp/{\sim},\distqsp) \to (\bbR^\embeddingdim,\norm{\cdot}_2)$ is \emph{not} injective at $f$.
\end{lemma}

\begin{proof}
    Let $(v_i)_{i \in \bbN},(w_i)_{i \in \bbN} \in F$ be two sequences such that $v_i \not\sim w_i$ for $i \in \bbN$, $\lim_{i \to \infty} v_i = \lim_{i \to \infty} w_i = f$ and 
    \begin{equation*}
        \lim_{i \to \infty} \frac{\norm{\embedding(v_i) - \embedding(w_i)}_2}{\distqsp(v_i,w_i)} = 0.
    \end{equation*}
    Expand the sequences into the orthogonal basis $(f_j)_{j = 1}^p$ of $F$:
    \begin{equation*}
        v_i = \sum_{j = 1}^p c_{ij} f_j, \qquad w_i = \sum_{j = 1}^p d_{ij} f_j,
    \end{equation*}
    where $c_{ij},d_{ij} \in \bbR$ for $i \in \bbN$, $j \in [p]$.

    \begin{proof}[Claim]\let\qed\relax
        For $i \in \bbN$ large enough, we have 
        \begin{equation*}
            \permsort{ \coorbit{\phi_\ell} v_i } = \permsort{ \coorbit{\phi_\ell} f } = \permsort{ \coorbit{\phi_\ell} w_i }, \qquad \ell \in [\nwin].
        \end{equation*}
    \end{proof}

    \begin{proof}[Proof of the claim]\let\qed\relax
        We only show the first equality. The second equality follows in the same way: fix $\ell \in [\nwin]$ arbitrary and note that with the notation
        \begin{equation*}
            \norm{\coorbit{\phi_\ell}}_\mathrm{op} = \max_{\substack{v \in \vsp\\ \normvsp{v} = 1}} \norm{ \coorbit{\phi_\ell} v }_2,
        \end{equation*}
        we may estimate 
        \begin{align*}
            \MoveEqLeft[3] \diffminmax[\Bigg]{ \sum_{j = 1}^p (c_{ij} - 1) \coorbit{\phi_\ell} f_j } = \norm[\Bigg]{ \bfD \sum_{j = 1}^p (c_{ij} - 1) \coorbit{\phi_\ell} f_j }_\infty \leq \sqrt{2} \cdot \norm[\Bigg]{ \sum_{j = 1}^p (c_{ij} - 1) \coorbit{\phi_\ell} f_j }_2 \\
            &\leq \sqrt{2} \cdot \sum_{j = 1}^p \abs{c_{ij} - 1} \norm{ \coorbit{\phi_\ell} f_j }_2 \leq \sqrt{2} \norm{\coorbit{\phi_\ell}}_\mathrm{op} \cdot \sum_{j = 1}^p \abs{c_{ij} - 1} \normvsp{ f_j } \\
            &\leq \sqrt{2d} \norm{\coorbit{\phi_\ell}}_\mathrm{op} \cdot \left( \sum_{j = 1}^p \abs{c_{ij} - 1}^2 \normvsp{ f_j }^2 \right)^{1/2} \\
            &= \sqrt{2d} \norm{\coorbit{\phi_\ell}}_\mathrm{op} \cdot \normvsp[\Bigg]{ \sum_{j = 1}^p (c_{ij} - 1) f_j } = \sqrt{2d} \norm{\coorbit{\phi_\ell}}_\mathrm{op} \cdot \normvsp{ v_i - f }.
        \end{align*}
        The right-hand side tends to zero as $i \to \infty$. Therefore, there exists $I_\ell \in \bbN$ such that for all $i \geq I_\ell$, 
        \begin{equation*}
            \diffminmax[\Bigg]{ \sum_{j = 1}^p (c_{ij} - 1) \coorbit{\phi_\ell} f_j } < \disthyperplane{ \coorbit{\phi_\ell} f }.
        \end{equation*}
        According to equation~\eqref{eq:matryoshka_doll} (and Lemma~\ref{lem:stability_of_sets:sec:D}), we can conclude that $\permsort{ \coorbit{\phi_\ell} v_i } = \permsort{ \coorbit{\phi_\ell} f }$. Setting $I := \max_{\ell \in [\nwin]} I_\ell$ allows us to conclude that, for all $i \geq I$ and all $\ell \in [\nwin]$, $\permsort{ \coorbit{\phi_\ell} v_i } = \permsort{ \coorbit{\phi_\ell} f }$ as desired.
    \end{proof}

    Now, let us pick $(\sigma_\ell)_{\ell = 1}^\nwin \in S_\ordgroup$ be such that $\sigma_\ell \in \permsort{ \coorbit{\phi_\ell} f }$ for $\ell \in [\nwin]$ and note that the claim implies that 
    \begin{align*}
        \MoveEqLeft[3] \norm{ \embedding(v_i) - \embedding(w_i) }_2 \\
        ={}& \norm*{ \alpha \left( \begin{pmatrix}
            \sort{ \kappa_{\phi_1} v_i } - \sort{ \kappa_{\phi_1} w_i } & \dots & \sort{ \kappa_{\phi_\nwin} v_i } - \sort{ \kappa_{\phi_\nwin} w_i }
        \end{pmatrix} \right) }_2 \\
        ={}& \norm*{ \alpha \begin{pmatrix}
            \sigma_1 \coorbit{\phi_1} (v_i-w_i) & \dots & \sigma_\nwin \coorbit{\phi_\nwin} (v_i-w_i)
        \end{pmatrix} }_2,
    \end{align*}
    for $i \in \bbN$ large enough. So, let us define $u_i := (v_i-w_i)/\normvsp{v_i-w_i}$. Since the closed unit ball in $F$ is compact, we can find a subsequence along which $u_i$ converges. Let us pass to this subsequence and denote its limit by $u$. Then, we have $\normvsp{u} = 1$ as well as
    \begin{align*}
        \MoveEqLeft[3] \norm*{ \alpha \begin{pmatrix}
            \sigma_1 \coorbit{\phi_1} u & \dots & \sigma_\nwin \coorbit{\phi_\nwin} u
        \end{pmatrix} }_2 = \lim_{i \to \infty} \norm*{ \alpha \begin{pmatrix}
            \sigma_1 \coorbit{\phi_1} u_i & \dots & \sigma_\nwin \coorbit{\phi_\nwin} u_i
        \end{pmatrix} }_2 \\
        &= \lim_{i \to \infty} \frac{\norm*{ \alpha \begin{pmatrix}
            \sigma_1 \coorbit{\phi_1} (v_i-w_i) & \dots & \sigma_\nwin \coorbit{\phi_\nwin} (v_i-w_i)
        \end{pmatrix} }_2}{\normvsp{v_i-w_i}} \\
        &\leq \lim_{i \to \infty} \frac{\norm{\embedding(v_i) - \embedding(w_i)}_2}{\distqsp(v_i,w_i)} = 0
    \end{align*}
    and thus $\alpha \begin{pmatrix} \sigma_1 \coorbit{\phi_1} u & \dots & \sigma_\nwin \coorbit{\phi_\nwin} u \end{pmatrix} = \bfnull_{\embeddingdim}$.

    We finally note that there exist arbitrarily small $\epsilon > 0$ such that $f \not\sim f + \epsilon u$: indeed, assume the opposite and we can find a sequence $(\epsilon_n)_{n \in \bbN} \in \bbR_+$ converging to zero and a group element $g \in \group$ such that $f = g f + \epsilon_n g u$ for all $n \in \bbN$ because $\group$ is finite. Therefore, $g$ stabilises $f$ which implies $\epsilon_n g u = 0$ for all $n \in \bbN$: a contradiction to $\normvsp{u} = 1$. So, let us pick such an $\epsilon$ which also satisfies \begin{equation*}
        \epsilon < \min_{\ell \in [\nwin]} \frac{\disthyperplane{\coorbit{\phi_\ell} f}}{\diffminmax{\coorbit{\phi_\ell} u}}.
    \end{equation*}
    Then, $\permsort{ \coorbit{\phi_\ell} (f + \epsilon u) } \subseteq \permsort{ \coorbit{\phi_\ell} f }$ for all $\ell \in [\nwin]$ according to Lemma~\ref{lem:simple_inclusion:sec:D} and thus picking $(\sigma_\ell)_{\ell = 1}^\nwin \in S_\ordgroup$ such that $\sigma_\ell \in \permsort{ \coorbit{\phi_\ell} (f + \epsilon u) }$ for all $\ell \in [\nwin]$ guarantees that 
    \begin{align*}
        \embedding(f + \epsilon u) &= \alpha \begin{pmatrix}
            \sort{\coorbit{\phi_1} (f + \epsilon u)} & \dots & \sort{\coorbit{\phi_\nwin} (f + \epsilon u)}
        \end{pmatrix} \\
        &= \alpha \begin{pmatrix}
            \sigma_1 \coorbit{\phi_1} (f + \epsilon u) & \dots & \sigma_\nwin \coorbit{\phi_\nwin} (f + \epsilon u)
        \end{pmatrix} \\
        &= \alpha \begin{pmatrix}
            \sigma_1 \coorbit{\phi_1} f & \dots & \sigma_\nwin \coorbit{\phi_\nwin} f
        \end{pmatrix} + \epsilon \cdot \alpha \begin{pmatrix}
            \sigma_1 \coorbit{\phi_1} u & \dots & \sigma_\nwin \coorbit{\phi_\nwin} u
        \end{pmatrix} \\
        &= \alpha \begin{pmatrix}
            \sigma_1 \coorbit{\phi_1} f & \dots & \sigma_\nwin \coorbit{\phi_\nwin} f
        \end{pmatrix} = \alpha \begin{pmatrix}
            \sort{\coorbit{\phi_1} f} & \dots & \sort{\coorbit{\phi_\nwin} f}
        \end{pmatrix} \\
        &= \embedding(f);
    \end{align*}
    i.e.~$\widehat \embedding$ is \emph{not} injective at $f$.
\end{proof}

We are now ready to prove the induction step (cf.~\cite[Lemma~2.13 on p.~19]{balan2023GI})

\begin{lemma}[Induction step]\label{lem:induction_step:sec:D}
    Let $p \in \bbN$, $p < d$, and let $(f_j)_{j = 1}^p \in \vsp \setminus \{0\}$ be an orthogonal sequence such that the local lower Lipschitz constant of $\widehat \embedding : (\vsp/{\sim},\distqsp) \to (\bbR^\embeddingdim,\norm{\cdot}_2)$ vanishes at $f := f_1 + f_2 + \dots + f_p$, $\lVert f_1 \rVert = 1$ and 
    \begin{gather}
        \diffminmax{ \coorbit{\phi_\ell} f_{k+1} } < \disthyperplane[\Bigg]{ \coorbit{\phi_\ell} \sum_{j=1}^k f_j }, \label{eq:ass_numerator} \\
        \normvsp{f_{k+1}} < \frac12 \cdot \min_{g \not\in \stabgroup{\sum_{j=1}^k f_j}} \normvsp[\Bigg]{ (g-e) \sum_{j=1}^k f_j }, \label{eq:ass_denominator}
    \end{gather}
    for all $k \in [p-1]$, $\ell \in [\nwin]$.
    
    If $\widehat \embedding$ is injective, then there exists $f_{p+1} \in (\operatorname{sp}(f_j)_{j = 1}^p)^\perp \setminus \{0\}$ such that the local lower Lipschitz constant of $\widehat \embedding$ vanishes at $f + f_{p+1}$ and 
    \begin{gather*}
        \diffminmax{ \coorbit{\phi_\ell} f_{p+1} } < \disthyperplane[\Bigg]{ \coorbit{\phi_\ell} \sum_{j=1}^p f_j }, \\
        \normvsp{f_{p+1}} < \frac12 \cdot \min_{g \not\in \stabgroup{\sum_{j=1}^p f_j}} \normvsp[\Bigg]{ (g-e) \sum_{j=1}^p f_j },
    \end{gather*}
    for all $\ell \in [\nwin]$.
\end{lemma}

\newcommand{\vorth}{\mathbb{v}}
\newcommand{\worth}{\mathbb{w}}
\newcommand{\uorth}{u}

\begin{proof}
    Let $(v_i)_{i = 1}^\infty, (w_i)_{i = 1}^\infty \in \vsp$ be such that $v_i \not\sim w_i$ for $i \in \bbN$, $\lim_{i \to \infty} v_i = \lim_{i \to \infty} w_i = f$ and 
    \begin{equation*}
        \lim_{i \to \infty} \frac{\norm{ \embedding(v_i) - \embedding(w_i) }_2}{\distqsp(v_i,w_i)} = 0.
    \end{equation*}
    Let $(\vorth_i)_{i = 1}^\infty,(\worth_i)_{i = 1}^\infty \in (\operatorname{sp}(f_j)_{j = 1}^p)^\perp$ denote the orthogonal projections of $(v_i)_{i = 1}^\infty$, $(w_i)_{i = 1}^\infty$ onto the orthogonal complement of $\operatorname{sp}(f_j)_{j = 1}^p$. We may note that $\vorth_i \neq 0$ or $\worth_i \neq 0$ for $i \in \bbN$ large enough: indeed, assume by contradiction that there exists a subsequence along which $\vorth_i = \worth_i = 0$. Then, the local lower Lipschitz constant of $\widehat \embedding|_{F} : (F/{\sim},\distqsp) \to (\bbR^\embeddingdim,\norm{\cdot}_2)$ vanishes at $f$, where $F = \operatorname{sp}(f_j)_{j = 1}^p \subseteq \vsp$. Lemma~\ref{lem:proofofmaintheorem} now implies that $\widehat \gamma$ is \emph{not} injective: a contradiction.
    
    In the following, we assume without loss of generality that $\normvsp{\vorth_i} \leq \normvsp{\worth_i} \neq 0$ for $i \in \bbN$ by switching the roles of $v_i$ and $w_i$ if necessary. Now, write 
    \begin{equation*}
        v_i = \sum_{j = 1}^p c_{ij} f_j + \vorth_i, \qquad w_i = \sum_{j = 1}^p d_{ij} f_j + \worth_i,
    \end{equation*}
    where $c_{ij}, d_{ij} \in \bbR$ for $i \in \bbN$, $j \in [p]$. Next, we define 
    \begin{equation*}
        u_i := w_i - v_i + \vorth_i = \sum_{j = 1}^p (d_{ij} - c_{ij}) f_j + \worth_i, \qquad i \in \bbN,
    \end{equation*}
    and note that $\normvsp{u_i} \geq \normvsp{\worth_i} > 0$ by the orthogonality of $(f_j)_{j = 1}^p$ and $\worth_i$. Finally, set 
    \begin{equation*}
        t_i := \frac{\epsilon}{\sqrt{2} \normvsp{u_i}} \cdot \min\left\{ \min_{\ell \in [\nwin]} \norm{\coorbit{\phi_\ell}}_\mathrm{op}^{-1} \disthyperplane{\coorbit{\phi_\ell} f} , \frac{1}{2 \sqrt{2}} \cdot \min_{g \not\in \stabgroup{f}} \normvsp{(g-e)f} \right\}
    \end{equation*}
    for $i \in \bbN$, where $\epsilon < 1$ and $\norm{\coorbit{\phi_\ell}}_\mathrm{op} = \max_{\normvsp{v}=1} \norm{\coorbit{\phi_\ell} v}_2$ denotes the operator norm of the coorbit $\coorbit{\phi_\ell} : \vsp \to \bbR^\ordgroup$ for $\ell \in [\nwin]$.

    We note that 
    \begin{gather*}
        \lim_{i \to \infty} \normvsp{\vorth_i}^2 \leq \Bigg( \lim_{i \to \infty} \sum_{j = 1}^p \abs{c_{ij}-1}^2 \normvsp{f_j}^2 + \normvsp{\vorth_i}^2 \Bigg) = \lim_{i \to \infty} \normvsp{v_i - f}^2 = 0, \\
        \lim_{i \to \infty} \normvsp{u_i} \leq \lim_{i \to \infty} \normvsp{v_i - w_i} + \lim_{i \to \infty} \normvsp{\vorth_i} = 0.
    \end{gather*}
    It follows that $t_i \to \infty$ as $i \to \infty$. So, let us assume that $t_i > 1$ for the rest of this proof. Finally, note that $\normvsp{ t_i u_i }$ is constant in $i \in \bbN$. It follows that $t_i u_i$ converges along a subsequence. Similarly, $\normvsp{ t_i \vorth_i } = t_i \normvsp{ \vorth_i } \leq t_i \normvsp{ u_i } = \normvsp{ t_i u_i }$ is upper bounded by a constant in $i \in \bbN$ such that $t_i \vorth_i$ converges along a subsequence as well. Passing to these subsequences, we may write $t_i u_i \to u \neq 0$ and $t_i \vorth_i \to f_{p+1}$. Note that $\normvsp{ f_{p+1} } = \lim_{i \to \infty} t_i \normvsp{ \vorth_i } \leq \lim_{i \to \infty} t_i \normvsp{ u_i } = \normvsp{ u }$.

    \begin{proof}[Claim~1]\let\qed\relax
        The sequences 
        \begin{equation*}
            v_i' = f + t_i \vorth_i, \qquad w_i' = f + t_i u_i, \qquad i \in \bbN,
        \end{equation*}
        achieve lower Lipschitz constant zero.
    \end{proof}

    \begin{proof}[Proof of Claim~1]\let\qed\relax
        The claim is proven in two steps. In the first step, we bound the denominator $\distqsp(v_i', w_i')$: consider $(g_i)_{i = 1}^\infty \in \group$ such that $\distqsp(v_i', w_i') = \normvsp{ g_i v_i' - w_i' }$. Since $\group$ is finite, there exists a group element $g \in \group$ which occurs infinitely often. Let us pass to a subsequence along which this is the case. We claim that $g \in \stabgroup{f}$: indeed, let $h \not\in \stabgroup{f}$ be arbitrary. Then, we have 
        \begin{align*}
            \normvsp{h v_i' - w_i'} &= \normvsp{(h-e)f + t_i(h \vorth_i - u_i)} \\
            &\geq \normvsp{(h-e)f} - t_i \normvsp{h \vorth_i - u_i} \\
            &\geq \min_{h \not\in \stabgroup{f}} \normvsp{(h-e)f} - t_i ( \normvsp{\vorth_i} + \normvsp{u_i} ) \\
            &\geq \min_{h \not\in \stabgroup{f}} \normvsp{(h-e)f} - 2 t_i \normvsp{u_i} \\
            &> 2 t_i \normvsp{u_i} \geq t_i \normvsp{\vorth_i - u_i} = \normvsp{v_i' - w_i'} \geq \operatorname{dist}(v_i',w_i'),
        \end{align*}
        by the definition of $t_i$. So, $g \in \stabgroup{f}$ which implies that
        \begin{equation*}
            \distqsp(v_i', w_i') = t_i \normvsp{g \vorth_i - u_i} = t_i \normvsp[\Bigg]{g \vorth_i - \sum_{j = 1}^p (d_{ij} - c_{ij}) f_j - \worth_i}.
        \end{equation*}
        Now, we would like to use that $\stabgroup{f} \subseteq \stabgroup{ \sum_{j = 1}^{p-1} f_j } \subseteq \dots \subseteq \stabsort{ f_1 }$ which follows from Lemma~\ref{lem:simple_inclusion_stabilisers} by inequality~\eqref{eq:ass_denominator}. Therefore, $g \in \stabgroup{f_j}$ for all $j \in [p]$ and we have 
        \begin{align*}
            \distqsp(v_i', w_i') &= t_i \normvsp[\Bigg]{g \vorth_i - \sum_{j = 1}^p(d_{ij} - c_{ij}) f_j - \worth_i} \\
            &= t_i \normvsp[\Bigg]{g \Bigg( \sum_{j = 1}^p c_{ij} f_j + \vorth_i \Bigg) - \Bigg( \sum_{j = 1}^p d_{ij} f_j + \worth_i \Bigg)} \\
            &= t_i \normvsp{g v_i - w_i} \geq t_i \distqsp(v_i, w_i).
        \end{align*}
        
        In the second step, we consider the numerator $\norm{\embedding(v_i') - \embedding(w_i')}_2$. For all $i \in \bbN$ and $\ell \in [\nwin]$, pick arbitrary $\sigma_{i\ell} \in \permsort{ \coorbit{\phi_\ell} v_i }$ and $\tau_{i\ell} \in \permsort{ \coorbit{\phi_\ell} w_i }$. Then, we have
        \begin{equation}\label{eq:numerator:lem:generalisedinduction}
            \begin{aligned}
                \MoveEqLeft[3] \norm{\embedding(v_i) - \embedding(w_i)}_2 \\
                &= \norm*{ \alpha \begin{pmatrix}
                    \sort{ \coorbit{\phi_1} v_i } - \sort{ \coorbit{\phi_1} w_i } & \dots & \sort{ \coorbit{\phi_\nwin} v_i } - \sort{ \coorbit{\phi_\nwin} w_i }
                \end{pmatrix} }_2 \\
                &= \norm*{ \alpha \begin{pmatrix}
                    \sigma_{i1} \coorbit{\phi_1} v_i - \tau_{i1} \coorbit{\phi_1} w_i & \dots & \sigma_{i\nwin} \coorbit{\phi_\nwin} v_i - \tau_{i\nwin} \coorbit{\phi_1\nwin} w_i
                \end{pmatrix} }_2 \\
                &= \left\lVert \alpha \left( \begin{matrix}
                    \sigma_{i1} \coorbit{\phi_1} \left( \sum_{j = 1}^p c_{ij} f_j + \vorth_i \right) - \tau_{i1} \coorbit{\phi_1} \left( \sum_{j = 1}^p d_{ij} f_j + \worth_i \right) \end{matrix} \right. \right. \\
                & \qquad \left. \left. \begin{matrix}
                    \dots & \sigma_{i\nwin} \coorbit{\phi_\nwin} \left( \sum_{j = 1}^p c_{ij} f_j + \vorth_i \right) - \tau_{i\nwin} \coorbit{\phi_\nwin} \left( \sum_{j = 1}^p d_{ij} f_j + \worth_i \right) \end{matrix} \right) \right\rVert_2.
            \end{aligned}
        \end{equation}
        Next, we want to use that $\sigma_{i\ell},\tau_{i,\phi} \in \permsort{ \coorbit{\phi_\ell} \sum_{j = 1}^p c_{ij} f_j }$ for all $\ell \in [\nwin]$ and all $i \in \bbN$ large enough.
    \end{proof}

    \begin{proof}[Claim~2]\let\qed\relax
        There exists $I \in \bbN$ such that, for all $i \geq I$ and all $\ell \in [\nwin]$, we have 
        \begin{equation}\label{eq:Xinclusion:lem:generalisedinduction}
            \permsort{ \coorbit{\phi_\ell} v_i } = \permsort{ \coorbit{\phi_\ell} (f + \vorth_i) } \subseteq \permsort{ \coorbit{\phi_\ell} f } = \permsort[\Bigg]{ \coorbit{\phi_\ell} \sum_{j = 1}^p c_{ij} f_j }.
        \end{equation}
    \end{proof}

    \begin{proof}[Proof of Claim~2]\let\qed\relax
        Let us fix $\ell \in [\nwin]$ arbitrary. The first equality follows from Lemma~\ref{lem:finally_done:sec:D} and inequality~\eqref{eq:ass_numerator} once 
        \begin{equation}\label{eq:inequalities_proof_claim21}
            \max\Bigg\{ \diffminmax[\Bigg]{ \coorbit{\phi_\ell} \sum_{j = 1}^p (c_{ij} - 1) f_j }, \diffminmax{ \coorbit{\phi_\ell} \vorth_i } \Bigg\} \leq \frac12 \cdot \disthyperplane{ \coorbit{\phi_\ell} f }.
        \end{equation}
        As part of the proof of Lemma~\ref{lem:proofofmaintheorem}, we had shown that 
        \begin{equation*}
            \lim_{i \to \infty} \diffminmax[\Bigg]{ \coorbit{\phi_\ell} \sum_{j = 1}^p (c_{ij} - 1) f_j } = 0.
        \end{equation*}
        Additionally, we can show that 
        \begin{equation*}
            \lim_{i \to \infty} \diffminmax{ \coorbit{\phi_\ell} \vorth_i } \leq \sqrt{2} \cdot \lim_{i \to \infty} \norm{ \coorbit{\phi_\ell} \vorth_i }_2 \leq \sqrt{2} \norm{\coorbit{\phi_\ell}}_\mathrm{op} \cdot \lim_{i \to \infty} \normvsp{\vorth_i} = 0.
        \end{equation*}
        It follows that there exists an integer $I_\ell \in \bbN$ such that both inequalities in equation~\eqref{eq:inequalities_proof_claim21} are satisfed with $i \geq I_\ell$.
        
        The second inclusion in equation~\ref{eq:Xinclusion:lem:generalisedinduction} follows from Lemma~\ref{lem:simple_inclusion:sec:D} due to inequality~\eqref{eq:inequalities_proof_claim21}. The final equality follows from Lemma~\ref{lem:stability_of_sets:sec:D} due to inequality~\eqref{eq:inequalities_proof_claim21} and inequality~\eqref{eq:ass_numerator}. Setting $I := \max_{\ell \in [\nwin]} I_\ell$ finishes the proof of Claim~2.
    \end{proof}

    \begin{proof}[Claim~3]\let\qed\relax
        There exists $I \in \bbN$ such that, for all $i \geq I$ and all $\ell \in [\nwin]$, we have 
        \begin{equation}\label{eq:Yinclusion:lem:generalisedinduction}
            \permsort{ \coorbit{\phi_\ell} w_i } = \permsort{ \coorbit{\phi_\ell} (f + u_i) } \subseteq \permsort{ \coorbit{\phi_\ell} f }.
        \end{equation}
    \end{proof}

    \begin{proof}[Proof of Claim~3]\let\qed\relax
        The proof is almost identical to that of Claim~2 once we realise that 
        \begin{equation*}
            w_i = \sum_{j = 1}^p c_{ij} f_j + u_i, \qquad i \in \bbN.
        \end{equation*}
        We will therefore omit it. 
    \end{proof}

    \begin{proof}[Proof of Claim~1 (continued)]\let\qed\relax
        Passing to the sequences starting at $I$, we reconsider equation~\eqref{eq:numerator:lem:generalisedinduction}: we have
        \begin{equation*}
            \sigma_{i\ell} \coorbit{\phi_\ell} \sum_{j = 1}^p c_{ij} f_j = \sort[\Bigg]{ \coorbit{\phi_\ell} \sum_{j = 1}^p c_{ij} f_j } = \tau_{i\ell} \coorbit{\phi_\ell} \sum_{j = 1}^p c_{ij} f_j
        \end{equation*}
        and thus
        \begin{align*}
            \MoveEqLeft[3] \norm{\embedding(v_i) - \embedding(w_i)}_2 \\
                &= \left\lVert \alpha \left( \begin{matrix}
                    \sigma_{i1} \coorbit{\phi_1} \left( \sum_{j = 1}^p c_{ij} f_j + \vorth_i \right) - \tau_{i1} \coorbit{\phi_1} \left( \sum_{j = 1}^p d_{ij} f_j + \worth_i \right) \end{matrix} \right. \right. \\
                & \qquad \left. \left. \begin{matrix}
                    \dots & \sigma_{i\nwin} \coorbit{\phi_\nwin} \left( \sum_{j = 1}^p c_{ij} f_j + \vorth_i \right) - \tau_{i\nwin} \coorbit{\phi_\nwin} \left( \sum_{j = 1}^p d_{ij} f_j + \worth_i \right) \end{matrix} \right) \right\rVert_2 \\
                &= \left\lVert \alpha \left( \begin{matrix}
                    \sigma_{i1} \coorbit{\phi_1} \vorth_i - \tau_{i1} \coorbit{\phi_1} \left( \sum_{j = 1}^p (d_{ij}-c_{ij}) f_j + \worth_i \right) \end{matrix} \right. \right. \\
                & \qquad \left. \left. \begin{matrix} \dots &
                    \sigma_{i\nwin} \coorbit{\phi_\nwin} \vorth_i - \tau_{i\nwin} \coorbit{\phi_\nwin} \left( \sum_{j = 1}^p (d_{ij}-c_{ij}) f_j + \worth_i \right) \end{matrix} \right) \right\rVert_2.
        \end{align*}
        Next, we may use that $\sigma_{i\ell}, \tau_{i\ell} \in \permsort{ \coorbit{\phi_\ell} f }$ for all $i \in \bbN$ and all $\ell \in [\nwin]$ to see that
        \begin{equation}\label{eq:numerator_again:lem:generalisedinduction}
            \begin{aligned}
                \MoveEqLeft[3] t_i \norm{\embedding(v_i) - \embedding(w_i)}_2 = t_i \norm*{ \alpha \begin{pmatrix}
                    \sigma_{i1} \coorbit{\phi_1} \vorth_i - \tau_{i1} \coorbit{\phi_1} u_i & \dots & \sigma_{i\nwin} \coorbit{\phi_\nwin} \vorth_i - \tau_{i\nwin} \coorbit{\phi_\nwin} u_i
                \end{pmatrix} }_2 \\
                &= \left\lVert \alpha \left( \begin{matrix}
                    \sigma_{i1} \coorbit{\phi_1} (f + t_i \vorth_i) - \tau_{i1} \coorbit{\phi_1} (f + t_iu_i) \end{matrix} \right. \right. \\
                & \qquad \left. \left. \begin{matrix} \dots & \sigma_{i\nwin} \coorbit{\phi_\nwin} (f + t_i \vorth_i) - \tau_{i\nwin} \coorbit{\phi_\nwin} (f + t_i u_i) \end{matrix} \right) \right\rVert_2 \\
                &= \norm*{ \alpha \begin{pmatrix}
                    \sigma_{i1} \coorbit{\phi_1} v_i' - \tau_{i1} \coorbit{\phi_1} w_i' & \dots & \sigma_{i\nwin} \coorbit{\phi_\nwin} v_i' - \tau_{i\nwin} \coorbit{\phi_\nwin} w_i'
                \end{pmatrix} }_2.
            \end{aligned}
        \end{equation}
        
        Finally, we want to use $\sigma_{i\ell} \in \permsort{\coorbit{\phi_\ell} v_i'}$ and $\tau_{i\ell} \in \permsort{\coorbit{\phi_\ell} w_i'}$ for all $i \in \bbN$ and all $\ell \in [\nwin]$. The first inclusion is true because $\permsort{\coorbit{\phi_\ell} (f+\vorth_i)} = \permsort{\coorbit{\phi_\ell} (f + t_i\vorth_i)}$ is guaranteed by Lemma~\ref{lem:equality_along_interval:sec:D} since
        \begin{align*}
            \diffminmax{\coorbit{\phi_\ell} t_i \vorth_i} &= t_i \diffminmax{\coorbit{\phi_\ell} \vorth_i} \leq \sqrt{2} t_i \norm{\coorbit{\phi_\ell} \vorth_i}_2 \leq \sqrt{2} \norm{\coorbit{\phi_\ell}}_\mathrm{op} t_i \normvsp{\vorth_i} \\
            &\leq \sqrt{2} \norm{\coorbit{\phi_\ell}}_\mathrm{op} t_i \normvsp{u_i} < \disthyperplane{ \coorbit{\phi_\ell} f }
        \end{align*}
        according to the definition of $t_i > 1$. The second inclusion follows in exactly the same way. 

        Returning to equation~\eqref{eq:numerator_again:lem:generalisedinduction}, we have 
        \begin{align*}
            \MoveEqLeft[3] t_i \norm{\embedding(v_i) - \embedding(w_i)}_2 = \norm*{ \alpha \begin{pmatrix}
                \sigma_{i1} \coorbit{\phi_1} v_i' - \tau_{i1} \coorbit{\phi_1} w_i' & \dots & \sigma_{i\nwin} \coorbit{\phi_\nwin} v_i' - \tau_{i\nwin} \coorbit{\phi_\nwin} w_i'
            \end{pmatrix} }_2 \\
            &= \norm*{ \alpha \begin{pmatrix}
                \sort{ \coorbit{\phi_1} v_i' } - \sort{ \coorbit{\phi_1} w_i' } & \dots & \sort{ \coorbit{\phi_\nwin} v_i' } - \sort{ \coorbit{\phi_\nwin} w_i' }
            \end{pmatrix} }_2 \\
            &= \norm{\embedding(v_i') - \embedding(w_i')}_2.
        \end{align*}
        We finally conclude that 
        \begin{equation*}
            \lim_{i \to \infty} \frac{\norm{\embedding(v_i') - \embedding(w_i')}_2}{\distqsp(v_i', w_i')} \leq \lim_{i \to \infty} \frac{\norm{\embedding(v_i) - \embedding(w_i)}_2}{\distqsp(v_i, w_i)} = 0.
        \end{equation*}
    \end{proof}

    Remember that $t_i \vorth_i$ and $t_i u_i$ are bounded sequences that converge to $f_{p+1}$ and $u$, respectively, as $i \to \infty$. Therefore, $v_i'$ and $w_i'$ are bounded sequences that converge to $f + f_{p+1}$ and $f + u$, respectively. Therefore, 
    \begin{align*}
        \norm{\embedding(f + f_{p+1}) - \embedding(f + u)}_2 &= \lim_{i \to \infty} \norm{\embedding(v_i') - \embedding(w_i')}_2 \\
        &\lesssim \lim_{i \to \infty} \frac{\norm{\embedding(v_i') - \embedding(w_i')}_2}{\distqsp(v_i', w_i')} = 0.
    \end{align*}
    Now, the injectivity of $\widehat \embedding$ implies that $f + f_{p+1} \sim f + u$; i.e., for some $g \in \group$, we have $g(f + f_{p+1}) = f + u$. It follows that $g \in \stabgroup{f}$: indeed, assume, by contradiction, that $g \not\in \stabgroup{f}$. Then, we have 
    \begin{align*}
        0 &= \normvsp{g(f + f_{p+1}) - (f + u)} \geq \normvsp{(g-e)f} - \normvsp{g f_{p+1} - u} \\
        &\geq \min_{g \not\in \stabgroup{f}} \normvsp{(g-e)f} - \normvsp{f_{p+1}} - \normvsp{u} \\
        &\geq \min_{g \not\in \stabgroup{f}} \normvsp{(g-e)f} - 2 \normvsp{u} > 0
    \end{align*}
    because
    \begin{equation*}
        \normvsp{u} = \lim_{i \to \infty} t_i \normvsp{u_i} \leq \frac{1}{4} \cdot \min_{g \not\in \stabgroup{f}} \normvsp{(g-e)f}.
    \end{equation*}
    Therefore, $g f_{p+1} = u$ and thus $\normvsp{f_{p+1}} = \normvsp{u} > 0$, which shows that $f_{p+1} \neq 0$. Finally, let $w_i'' := g^{-1} w_i' = f + t_i g^{-1} u_i$. Then, the sequences $(v_i')_{i = 1}^\infty$, $(w_i'')_{i = 1}^\infty$ still achieve lower Lipschitz constant zero. Therefore, they achieve local lower Lipschitz constant zero at $f + f_{p+1}$. Now, $f_{p+1}$ is orthogonal to $(f_j)_{j = 1}^k$ and satisfies 
    \begin{align*}
        \normvsp{f_{p+1}} &= \lim_{i \to \infty} t_i \normvsp{\vorth_i} \leq \lim_{i \to \infty} t_i \normvsp{u_i} \\
        &= \frac{\epsilon}{\sqrt{2}} \cdot \min\left\{ \min_{\ell \in [\nwin]} \norm{\coorbit{\phi_\ell}}_\mathrm{op}^{-1} \disthyperplane{\coorbit{\phi_\ell} f} , \frac{1}{2 \sqrt{2}} \cdot \min_{g \not\in \stabgroup{f}} \normvsp{(g-e)f} \right\}
    \end{align*}
    Therefore, 
    \begin{align*}
        \diffminmax{ \coorbit{\phi_\ell} f_{p+1} } &\leq \sqrt{2} \cdot \norm{ \coorbit{\phi_\ell} f_{p+1} }_2 \leq \sqrt{2} \norm{\coorbit{\phi_\ell}}_\mathrm{op} \cdot \normvsp{f_{p+1}} \\
        &< \norm{\coorbit{\phi_\ell}}_\mathrm{op} \cdot \min_{k \in [\nwin]} \norm{\coorbit{\phi_k}}_\mathrm{op}^{-1} \disthyperplane{\coorbit{\phi_{k}} f} \leq \disthyperplane{\coorbit{\phi_\ell} f}
    \end{align*}
    and the lemma is proven.
\end{proof}

Combining the base case and the induction step allows us to prove that injectivity implies bi-Lipschitz.

\begin{proof}[Proof of Theorem~\ref{thm:main}]
    Remember that $\widehat \embedding$ is Lipschitz continuous according to Proposition~\ref{prop:Lipschitz:sec:D}. Now, assume by contradiction that $\widehat \embedding$ is \emph{not} lower Lipschitz continuous. Then, Lemma~\ref{lem:base_case:sec:D} together with Lemma~\ref{lem:induction_step:sec:D} show that there exists an orthogonal basis $(f_j)_{j = 1}^\dimvsp \in \vsp$ such that the local lower Lipschitz constant of $\widehat \embedding$ vanishes at $f := f_1 + f_2 + \dots + f_\dimvsp$ and 
    \begin{equation*}
        \forall \ell \in [\nwin], k \in [d-1] : \diffminmax{ \coorbit{ \phi_\ell } f_{k+1} } < \disthyperplane[\Bigg]{ \sum_{j = 1}^k \coorbit{\phi_\ell} f_j }.
    \end{equation*}
    It follows from Lemma~\ref{lem:proofofmaintheorem} that $\widehat \gamma$ is \emph{not} injective: a contradiction.
\end{proof}

\bibliography{sources}
\bibliographystyle{plain}

\end{document}